\documentclass[11pt]{amsart}
\usepackage{amssymb}
\usepackage{amscd} 
\usepackage{amsfonts, amsmath, amsthm}
\usepackage{tikz-cd}
\usepackage{url}
\usepackage[a4paper, margin=1.0in]{geometry}
\usepackage{mathtools, mathrsfs, color} 
\usepackage{graphicx}
\usepackage{tikz,tikz-cd}
\usepackage[all,cmtip]{xy}
\numberwithin{equation}{section}
\usepackage{blkarray}
\usepackage{comment}
\usepackage[nobysame,alphabetic,initials,msc-links,lite]{amsrefs}
\usepackage{hyperref}
\usepackage{xcolor}
\usepackage[normalem]{ulem}
\allowdisplaybreaks

\definecolor{contentsred}{RGB}{145,0,0}
\makeatletter
\newif\ifNT@tocsubsections
\newlength{\NT@tocindent}
\newcommand{\NT@toclink}[1]{%
  \textcolor{contentsred}{\hyperlink{\NT@tocanchor}{#1}}}
\newcommand{\NT@tocsection}[2]{%
  \par\addvspace{6pt}%
  \NT@tocsubsectionsfalse
  \begingroup
    \normalsize\bfseries
    \leftskip\NT@tocindent
    \rightskip2em \parfillskip-2em \parindent\z@
    \noindent\hspace*{-\NT@tocindent}%
    \makebox[\NT@tocindent][r]{%
      \@ifnotempty{#1}{#1.}\hspace{.65em}}%
    \NT@toclink{#2}\nobreak
    \leaders\hbox{\normalfont\kern.15em.\kern.15em}\hfill
    \nobreak\makebox[2em][r]{%
      \NT@toclink{\NT@tocpage}}\par
  \endgroup
}
\newcommand{\NT@tocsubsection}[1]{%
  \@ifnotempty{#1}{%
    \ifNT@tocsubsections
      \unskip\nobreak\hspace{.4em}\textbullet\hspace{.4em}%
    \else
      \small\leftskip\NT@tocindent \rightskip2em plus1em
      \parfillskip\z@ plus1fil \parindent\z@
      \noindent\NT@tocsubsectionstrue
    \fi
    \NT@toclink{#1}%
  }%
}
\renewcommand{\tableofcontents}{%
  \par\addvspace{1.25\baselineskip}%
  \noindent\begin{minipage}{\textwidth}
    \hypersetup{pdfborder={0 0 0}}%
    \parindent\z@ \parskip\z@
    \hrule height .5pt
    \vspace{.4\baselineskip}
    {\large\bfseries\contentsname\par}
    \NT@tocsubsectionsfalse
    \def\contentsline##1##2##3##4{%
      \def\NT@tocpage{##3}\def\NT@tocanchor{##4}##2}%
    \def\tocsection##1##2##3{\NT@tocsection{##2}{##3}}%
    \let\tocappendix\tocsection
    \def\tocsubsection##1##2##3{\NT@tocsubsection{##3}}%
    \def\tocsubsubsection##1##2##3{}%
    \@input{\jobname.toc}%
    \par\vspace{.65\baselineskip}
    \hrule height .5pt
    \if@filesw
      \@ifundefined{tf@toc}{\newwrite\tf@toc}{}%
      \immediate\openout\tf@toc\jobname.toc\relax
    \fi
  \end{minipage}\par
}
\makeatother

\def\today{\number\day\space\ifcase\month\or   January\or February\or
   March\or April\or May\or June\or   July\or August\or September\or
   October\or November\or December\fi\   \number\year}

\newtheorem{thm}{Theorem}[section]
\newtheorem{lem}[thm]{Lemma}
\newtheorem{prp}[thm]{Proposition}
\newtheorem{dfn}[thm]{Definition}
\newtheorem{cor}[thm]{Corollary}

\newtheorem{rmk}[thm]{Remark}

\newtheorem{exa}[thm]{Example}

\newcommand{\beq}{\begin{equation}}
\newcommand{\eeq}{\end{equation}}
\newcommand{\beqr}{\begin{eqnarray*}}
\newcommand{\eeqr}{\end{eqnarray*}}
\newcommand{\bal}{\begin{align*}}
\newcommand{\eal}{\end{align*}}
\newcommand{\bei}{\begin{itemize}}
\newcommand{\eei}{\end{itemize}}

\newcommand{\af}{\alpha}
\newcommand{\bt}{\beta}
\newcommand{\gm}{\gamma}
\newcommand{\dt}{\delta}
\newcommand{\ep}{\varepsilon}

\newcommand{\ch}{\chi}

\newcommand{\te}{\theta}
\newcommand{\ld}{\lambda}
\newcommand{\sm}{\sigma}

\newcommand{\ps}{\psi}
\newcommand{\rh}{\rho}
\newcommand{\om}{\omega}
\newcommand{\ta}{\tau}

\newcommand{\Gm}{\Gamma}

\newcommand{\Th}{\Theta}
\newcommand{\Ld}{\Lambda}

\newcommand{\Ph}{\Phi}
\newcommand{\Ps}{\Psi}

\newcommand{\Q}{{\mathbb{Q}}}
\newcommand{\Z}{{\mathbb{Z}}}
\newcommand{\R}{{\mathbb{R}}}
\newcommand{\C}{{\mathbb{C}}}

\newcommand{\T}{{\mathbb{T}}}

\newcommand{\pf}{{\operatorname{pf}}}

\newcommand{\Id}{{\operatorname{Id}}}

\newcommand{\diag}{{\mathrm{diag}}}

\newcommand{\rank}{{\mathrm{rank}}}

\newcommand{\Aut}{{\mathrm{Aut}}}

\newcommand{\GL}{{\mathrm{GL}}}

\newcommand{\I}{\infty}

\newcommand{\A}{\mathscr{A}}

\title{Equivariant Morita equivalence for finite cyclic group actions on noncommutative tori}

\author[]{Sayan Chakraborty} 
\address{ Institute for Advancing Intelligence, TCG CREST, Sector V, Salt Lake, Kolkata 700091, India.}
\address{\href{mailto:sayan.chakraborty@tcgcrest.org}{sayan.chakraborty@tcgcrest.org}}
\author[]{Pratik Kumar Kundu} 
\address{ Institute for Advancing Intelligence, TCG CREST, Sector V, Salt Lake, Kolkata 700091, India.}
\address{\href{mailto:pratik.kundu.79@tcgcrest.org}{pratik.kundu.79@tcgcrest.org}}

\keywords{Morita equivalence, noncommutative torus, crossed product}
\subjclass[2020]{46L35, 46L55}

\begin{document}

\begin{abstract}

Li proved that the noncommutative tori $A_\theta$ and
$A_{g\theta}$ are strongly Morita equivalent whenever
$g\in\rm{SO}(n,n|\Z)$ and $g\theta$ is defined. We study
equivariant extensions of these equivalences for canonical finite
cyclic group actions. Let
$W\in\GL_n(\Z)$ satisfy $W^t\theta W=\theta$, $W^k=I_n$
and $1\notin\operatorname{spec}(W)$, and write
$\rho(R)=\operatorname{diag}(R,(R^t)^{-1})$. If
$g\rho(W^t)g^{-1}=\rho(P)$ for some $P\in\GL_n(\Z)$,
we equip Li's $A_{g\theta}$--$A_\theta$ Heisenberg
equivalence bimodule with a compatible $\Z_k$-action.
Consequently, the crossed products
$A_{g\theta}\rtimes_{\alpha_{P^t}}\Z_k$ and
$A_\theta\rtimes_{\alpha_W}\Z_k$
are strongly Morita equivalent.
For the canonical flip action by $\Z_2$, we prove that Morita equivalence of crossed
products implies Morita equivalence of the underlying tori in every
dimension, and prove the converse in dimensions two through four.
For each fixed dimension $2\leq n\leq4$, the complete invariants are the central rank
and canonical $K_0$-trace range of the underlying torus, with the
latter taken up to positive scaling. For the standard cyclic group
actions of orders $2$, $3$, $4$ and $6$ on rotation algebras,
we prove that the crossed products are Morita equivalent exactly
when the action orders agree and the parameters lie in the same
fractional-linear $\rm{GL}_2(\Z)$-orbit, generalizing the result
of B\"onicke et al.\ from irrational
to arbitrary real parameters. We also construct infinite families of distinct
parameters giving Morita-equivalent crossed products for each
fixed $W$ satisfying the stated hypotheses.
\end{abstract}
\maketitle \pagestyle{myheadings} \markboth{Sayan~Chakraborty and Pratik~Kumar~Kundu}{Equivariant Morita equivalence for finite cyclic group actions on noncommutative tori}
\vspace{1.5cm}
\tableofcontents
\clearpage

\section{Introduction}
Let $\mathcal T_n(\R)$ denote the set of real skew-symmetric $n\times n$
matrices. For $\theta\in\mathcal T_n(\R)$, we use the Weyl
presentation of the noncommutative torus $A_\theta$: it is the
universal $\mathrm{C}^*$-algebra generated by unitaries
$\{U_x^\theta:x\in\Z^n\}$ satisfying
\[
 U_x^\theta U_y^\theta
 =\sigma_\theta(x,y)U_{x+y}^\theta,
 \qquad
 \sigma_\theta(x,y)=\exp(\pi i x^t\theta y),
 \qquad
 (U_x^\theta)^*=U_{-x}^\theta.
\]
A matrix $W\in\GL_n(\Z)$ satisfying $W^t\theta W=\theta$
induces the canonical automorphism
\[
 \alpha_W(U_x^\theta)=U_{Wx}^\theta.
\]
If $W^k=I_n$, this automorphism defines an action of $\Z_k$.

Rieffel classified irrational rotation algebras up to strong
Morita equivalence \cite{Rieffel1981} and constructed Heisenberg
modules over higher-dimensional noncommutative tori
\cite{Rieffel1988}. Rieffel and Schwarz subsequently introduced
the partial action of $\rm{SO}(n,n|\Z)$ on $\mathcal T_n(\R)$
and established its relation to Morita equivalence \cite{RSch}.
Here
\[
 \rm{SO}(n,n|\Z)
 =\{g\in\rm{SL}_{2n}(\Z):g^tJ_{n,n}g=J_{n,n}\},
 \qquad
 J_{n,n}=\begin{pmatrix}0&I_n\\I_n&0\end{pmatrix},
\]
and, for $g=\begin{psmallmatrix}A&B\\C&D\end{psmallmatrix}$,
the action is given by
\[
 g\theta=(A\theta+B)(C\theta+D)^{-1}
\]
whenever $C\theta+D$ is invertible. Li proved that
$A_{g\theta}\sim_M A_\theta$ for every such pair $(g,\theta)$,
by constructing an explicit Heisenberg equivalence bimodule
\cite{Li2004}*{Theorem~1.1}. Throughout, $\sim_M$ denotes
strong Morita equivalence of $\mathrm{C}^*$-algebras. We write
\[
 \rho(R)=\begin{pmatrix}R&0\\0&(R^t)^{-1}\end{pmatrix},
 \qquad
 \nu(N)=\begin{pmatrix}I_n&N\\0&I_n\end{pmatrix},
\]
for $R\in\GL_n(\Z)$ and integral skew-symmetric $N$. Both
$\rho(R)$ and $\nu(N)$ belong to $\rm{SO}(n,n|\Z)$ and
act by $\rho(R)\theta=R\theta R^t$ and
$\nu(N)\theta=\theta+N$. Notice
that the noncommutative tori corresponding to the matrices $\rho(R)\theta$ and $\nu(N)$ are
both isomorphic to $A_{\theta}$.

For finite cyclic group actions generated by $\alpha_W$, Chakraborty and Luef used metaplectic
operators to equip Heisenberg right modules with compatible group
actions \cite{CL}. Chakraborty extended this construction to
fundamental projective modules induced by certain elements of
$\rm{SO}(n,n|\Z)$ and identified a conjugacy condition
for lifting the associated Morita equivalences
\cite{ChTrace}*{Theorem~4.2 and Remark~4.3}. We treat arbitrary
$g\in\rm{SO}(n,n|\Z)$ satisfying that condition, using Li's
full construction.
We verify covariance of both module actions and both inner
products, obtaining the following equivariant imprimitivity theorem.

\begin{thm}\label{thm:main}
Let $\theta\in\mathcal T_n(\R)$ and
$g\in\rm{SO}(n,n|\Z)$, with $g\theta$ defined. Let
$W\in\GL_n(\Z)$ and an integer $k\geq2$ satisfy
\begin{equation}\label{eq:hypW}
 W^t\theta W=\theta,\qquad W^k=I_n,\qquad
 1\notin\operatorname{spec}(W).
\end{equation}
Suppose that $P\in\GL_n(\Z)$ satisfies
\begin{equation}\label{eq:conjugacy}
 g\rho(W^t)g^{-1}=\rho(P).
\end{equation}
Then $P^k=I_n$, $P(g\theta)P^t=g\theta$, and there is a
$\Z_k$-equivariant $A_{g\theta}$--$A_\theta$ imprimitivity
bimodule for the actions $\alpha_{P^t}$ and $\alpha_W$.
Consequently,
\[
 A_{g\theta}\rtimes_{\alpha_{P^t}}\Z_k
 \sim_M A_\theta\rtimes_{\alpha_W}\Z_k.
\]
\end{thm}

The conjugacy condition ensures that the transformed action is
again induced by an integral lattice automorphism. To construct
the bimodule action, we construct a unitary implementer $\mathcal U$
of a symplectic automorphism of the locally compact abelian group
$M\times\widehat M$ in Li's construction. The hypothesis
$1\notin\operatorname{spec}(W)$ allows us to correct the implementer
so that it satisfies $\mathcal U^k=I$ and the required covariance identities.
The crossed-product conclusion then follows
from the equivariant imprimitivity theorem of Combes and
Curto--Muhly--Williams \cites{Combes1984,CMW1984}.

For the canonical flip, put
\[
 \alpha_\theta=\alpha_{-I_n},\qquad
 B_\theta=A_\theta\rtimes_{\alpha_\theta}\Z_2.
\]
Since $\rho(-I_n)=-I_{2n}$ is central, Theorem~\ref{thm:main}
gives $B_{g\theta}\sim_M B_\theta$ whenever $g\theta$ is
defined. For simple tori, Echterhoff, L\"uck, Phillips and
Walters proved that the flip crossed products are AF algebras
\cite{ELPW}*{Theorem~6.6}, and Chakraborty proved that
$B_\theta\sim_M B_{\theta'}$ if and only if
$A_\theta\sim_M A_{\theta'}$
\cite{ChTrace}*{Corollary~5.11}. We extend the implication from
crossed products to underlying tori to all parameters. We also
prove the converse under the rank conditions below.

The invariants are those occurring in Elliott and Li's
classification of noncommutative tori by their centres and
ordered $K_0$-groups \cite{EL}*{Theorem~1.1}. Set
\begin{equation}\label{eq:invariants}
 \begin{split}
 H_\theta&=\{x\in\Z^n:\theta x\in\Z^n\},\qquad
 r_\theta=\rank H_\theta,\qquad m_\theta=n-r_\theta,\\
 \mathcal R_\theta&=(\tau_\theta)_*K_0(A_\theta)
 =\sum_{\substack{I\subseteq\{1,\ldots,n\}\\|I|\text{ even}}}\pf(\theta_I)\Z.
 \end{split}
\end{equation}
Here $\tau_\theta(U_x^\theta)=\delta_{x,0}$ is the canonical
trace, $\theta_I$ is the principal submatrix indexed by $I$,
and $\pf(\theta_\varnothing)=1$. The trace-range formula is
due to Elliott \cite{Elliott1984}. The centre of $A_\theta$ is
isomorphic to $C(\T^{r_\theta})$, and $A_\theta$ is simple
exactly when $H_\theta=0$; this is the meaning of
nondegeneracy used here. Elliott--Li's reduction gives
$A_\theta\sim_M C(\T^{r_\theta})\otimes A_\eta$ for a
nondegenerate $m_\theta$-dimensional parameter $\eta$
\cite{EL}*{Proposition~3.3}, with $A_\eta=\C$ when
$m_\theta=0$.

\begin{thm}\label{thm:classification}
Fix $n\geq2$ and let $\theta,\theta'\in\mathcal T_n(\R)$.
\begin{enumerate}
\renewcommand{\theenumi}{\roman{enumi}}
\renewcommand{\labelenumi}{\textup{(\theenumi)}}
\item If $B_\theta\sim_M B_{\theta'}$, then
$A_\theta\sim_M A_{\theta'}$.
\item Put
\[
 \mathcal C_n=\{\eta\in\mathcal T_n(\R):
 r_\eta\leq1\text{ or }m_\eta\in\{0,2\}\}.
\]
If $\theta,\theta'\in\mathcal C_n$, the following conditions
are equivalent:
\begin{align*}
 &B_\theta\sim_M B_{\theta'};\\
 &A_\theta\sim_M A_{\theta'};\\
 &r_\theta=r_{\theta'}\quad\text{and}\quad
 \mathcal R_{\theta'}=\lambda\mathcal R_\theta
 \text{ for some }\lambda>0.
\end{align*}
\end{enumerate}
In particular, \textup{(ii)} applies to every pair of parameters
in dimensions $2$, $3$ and $4$.
\end{thm}

For rotation algebras, i.e., in the case $n=2$, we consider the standard actions
associated to the nontrivial finite cyclic subgroups of
$\rm{SL}_2(\Z)$, whose orders are $2,3,4,6$. Write $A_t$
for the rotation algebra with parameter $t\in\R$ and
$C_{t,k}$ for its crossed product by the standard action of
order $k$. For irrational parameters, these algebras are AF
by \cite{ELPW}, and their Morita equivalence classes were
determined in \cite{BCHL}*{Theorem~5.3}. We obtain the
classification for all real parameters
(Theorem~\ref{thm:two-dimensional-classification}):
\[
 C_{t,k}\sim_M C_{s,l}
 \quad\Longleftrightarrow\quad
 k=l\ \text{ and }\ s=\frac{at+b}{ct+d}
 \ \text{ for some }\
 \begin{pmatrix}a&b\\c&d\end{pmatrix}\in\rm{GL}_2(\Z),
\]
where $ct+d\neq0$. In particular, for each fixed order, all
rational parameters give Morita-equivalent crossed products.

Finally, Theorem~\ref{thm:infinite-fixed-action} constructs
infinitely many distinct $W$-invariant parameters whose crossed
products by the actions induced by $W$ are Morita equivalent,
for every $(\theta,W)$ satisfying \eqref{eq:hypW} and $n\geq2$.
For $\theta\neq0$, we give an explicit integral skew-symmetric
matrix $C$ with $WCW^t=C$ such that
\[
 \theta_j=\theta(I_n+jC\theta)^{-1}
\]
is defined and gives pairwise distinct parameters for all but
finitely many $j\in\Z$. For $\theta=0$, invariant integral
shifts provide such a family. The examples include the
four-dimensional actions of orders $5,8,10,12$ constructed by
Jeong and Lee \cite{JL}*{Proposition~6.4}.

Section~2 recalls the Heisenberg representation and Li's
bimodules. Sections~\ref{sec:symplectic-construction}
and~\ref{sec:equivariant-bimodule} construct the symplectic
automorphism and prove Theorem~\ref{thm:main}. Section~5
contains the classifications and examples; the detailed
computations are collected in
Appendix~\ref{app:symplectic-computations}.

\let\NTsavedaddcontentsline\addcontentsline
\renewcommand{\addcontentsline}[3]{}
\section*{Acknowledgements}
\let\addcontentsline\NTsavedaddcontentsline
We thank Ingo Ullisch for the initial discussions on this project.
The authors acknowledge the use of ChatGPT, developed by OpenAI,
for assistance with language editing and improving the presentation of
parts of this manuscript. The research of the second named author was
supported by TCG CREST Ph.D Fellowship.

\section{Preliminaries}\label{Prelim}

\subsection{Heisenberg--Weyl operators}

We follow Rieffel's construction of Heisenberg--Weyl operators in \cite{Rieffel1988}*{Section~2}.

Let $M$ be a locally compact abelian group and let $\widehat M$ be its Pontryagin dual. Put $G=M\times\widehat M$. For $z=(m,\chi)$ and $w=(n,\psi)\in G$, set
\[
\beta(z,w)=\langle m,\psi\rangle :=\psi(m),\qquad
\rho_G(z,w)=\beta(z,w)\overline{\beta(w,z)}.
\]
A topological automorphism $Q$ of $G$ is called \emph{symplectic} if
\[
\rho_G(Qz,Qw)=\rho_G(z,w)\qquad\text{for all }z,w\in G.
\]
We recall the following self-duality property of $G$; see \cite{Rieffel1988}*{Section~2, p.~264}.
\begin{lem}\label{lem:rhoG-self-duality}
The map $\rho_G$ is a continuous alternating bicharacter on $G$. Moreover, the map
\[
\Xi:G\longrightarrow \widehat{G},
\qquad
\Xi(z)(w):=\rho_G(z,w),
\]
is a topological group isomorphism. In particular, $\rho_G$ is
nondegenerate.

Equivalently, for every character $\varphi\in\widehat{G}$, there exists a
unique $z_\varphi\in G$ such that
\[
\varphi(w)=\rho_G(z_\varphi,w)
\qquad
\text{for all }w\in G.
\]
\end{lem}

For $(m,\chi)\in G$, the Heisenberg--Weyl operator $\pi(m,\chi)$ on $L^2(M)$ is defined by
\[
(\pi(m,\chi)f)(r)=\langle r,\chi \rangle f(r-m).
\]
These operators form a projective representation $\pi$ of $G$ satisfying 
\[
\pi(z)\pi(w)=\overline{\bt(z,w)}\pi(z+w),\qquad \pi(z)\pi(w)=\overline{\rh_G(z,w)}\pi(w)\pi(z).
\]
We shall use the irreducibility of this Heisenberg--Weyl representation in the form of the Stone--von Neumann theorem for locally compact abelian groups; see \cites{Segal1963, Weil1964}. We record three facts that will be used later.
\begin{lem}\label{lem:scalar-operator}
If $\pi(z)$ is a scalar operator, then $z=0$.
\end{lem}

\begin{proof}
Suppose that $\pi(z)=\lambda I$ for some $\lambda\in\mathbb T$. Then $\pi(z)$ commutes with
$\pi(w)$ for every $w\in G$, and hence
\[
\pi(z)\pi(w)=\pi(w)\pi(z).
\]
The commutation relation gives $\rho_G(z,w)=1$ for every
$w\in G$. By Lemma~\ref{lem:rhoG-self-duality}, $z=0$.
\end{proof}
\begin{lem}\label{lem:character-extension}
    Let $H$ be a closed subgroup of a locally compact abelian group $G$.  Every continuous character of $H$ extends to a continuous character of $G$.
\end{lem}

\begin{proof}
    This is the character-extension theorem in Pontryagin duality; see~\cite{HewittRoss1979}*{Theorem~24.12}. Equivalently, the restriction homomorphism $\widehat G\to\widehat H$ is surjective.
\end{proof}

For $M=\R^p\times\Z^q\times F$, with $F$ a finite abelian group, let
$\mathcal S(M)$ be the space of functions smooth in the real
variables for which
\[
 \sup_{\substack{x\in\R^p,\ u\in\Z^q\\ a\in F}}
 (1+\|x\|+\|u\|)^N
 \bigl|\partial_x^\alpha f(x,u,a)\bigr|<\infty
\]
for every $N\geq0$ and every multi-index $\alpha$.
The space $\mathcal S(M)$ is called the Schwartz--Bruhat space of $M$.
we shall refer to it simply as the Schwartz space.

\begin{lem}\label{lem:unitary implementation}
Let $M=\R^p\times\Z^q\times F$, where $F$ is finite abelian, and put
$G=M\times\widehat M$. If $Q\in\Aut(G)$ is a topological
automorphism satisfying
\[
 \rho_G(Qz,Qw)=\rho_G(z,w)\qquad(z,w\in G),
\]
then there are a unitary operator $V_Q$ on $L^2(M)$ and a
continuous function $\eta_Q:G\to\T$ such that
\[
 V_Q\pi(z)V_Q^*=\eta_Q(z)\pi(Qz)\qquad(z\in G).
\]
Moreover, $V_Q$ preserves the Schwartz space: $V_Q\mathcal S(M)=\mathcal S(M)$.
\end{lem}

\begin{proof}[Proof sketch]
This is a well-known consequence of Weil's theorem
and his invariance result for the Schwartz--Bruhat space~\cite{Weil1964}, recalled
below. We give the argument in our conventions.
Put $\mu(z,w)=\overline{\beta(z,w)}$. Since $Q$ preserves
$\rho_G$, the quotient $\mu(Qz,Qw)/\mu(z,w)$ is a symmetric
continuous bicharacter. By \cite{Kleppner1965}*{Lemma~7.2} and \cite{DigernesVaradarajan2004}*{Lemma~3(1)} respectively,
there exists a continuous function
$\eta_Q:G\to\mathbb{T}$ such that such that
\begin{equation}\label{eq:unitary-implementation-phase}
 \frac{\eta_Q(z+w)}{\eta_Q(z)\eta_Q(w)}
 =\frac{\mu(Qz,Qw)}{\mu(z,w)}.
\end{equation}

Consider the Heisenberg group $H_\mu=G\times\T$ with product
$(z,t)(w,u)=(z+w,tu\mu(z,w))$, represented on $L^2(M)$ by
$(z,t)\mapsto t\pi(z)$. The phase equation makes
$(z,t)\mapsto(Qz,\eta_Q(z)t)$ an automorphism of $H_\mu$
fixing its centre pointwise. Weil's implementation theorem
therefore gives a unitary $V_Q$ satisfying the asserted
covariance relation; see~\cite{Weil1964}*{Chapter~I, no.~10,
Theorem~1 and equation~(15), p.~157}. Our convention is
$\pi(m,\chi)=U(-m,\chi)$ in Weil's notation.

Finally, Weil proves that these unitary normalizers preserve
the Schwartz--Bruhat space; see~\cite{Weil1964}*{Chapter~I,
nos.~11--12, pp.~157--159}. For the elementary group
$M=\R^p\times\Z^q\times F$, this space is precisely
$\mathcal S(M)$ defined above. Thus
$V_Q\mathcal S(M)=\mathcal S(M)$.
\end{proof}

\subsection{Li's Heisenberg bimodule}
We recall Li's construction of Morita equivalences between $A_\te$ and $A_{g\te}$ for $g\in\rm{SO}(n,n|\Z)$ with $g\te$ defined \cite{Li2004}. We construct embedding maps in the sense of Definition~\ref{def:embedding-map} to define the bimodule structure on $\mathcal S(M)$. From now on, $\mathcal T_n$ abbreviates $\mathcal T_n(\R)$.
We start with the following lemma:
\begin{lem}\label{special from g}
Let $g=\begin{pmatrix}
    A & B\\
    C & D\\
\end{pmatrix}\in\rm{SO}(n,n|\Z)$, and let $p\in\Z_{\geq0}$ satisfy $2p\leq n$. Put $q=n-2p$. Then the following are equivalent:
\begin{enumerate}
    \item There is some $\te\in\mathcal{T}_n$ for which $C\te+D$ is invertible and $(C\te+D)^{-1}C=\begin{pmatrix}
        F_{11} & 0\\
        0 & 0\\
    \end{pmatrix}$ for some $F_{11}\in\rm{GL}_{2p}(\R)$.
    \item There is a unique $Z\in\mathcal{T}_{2p}\cap \rm{M}_{2p}(\Q)$ such that
    \[
    C=\begin{pmatrix}
        C_{11} & 0\\
        C_{21} & 0\\
    \end{pmatrix},\quad D=\begin{pmatrix}
        -C_{11}Z & D_{12}\\
        -C_{21}Z & D_{22}\\
    \end{pmatrix}
    \] where $C_{11}\in\rm{M}_{2p}(\Z)$, $C_{21}\in\rm{M}_{q\times2p}(\Z)$, $D_{12}\in\rm{M}_{2p\times q}(\Z)$ and $D_{22}\in\rm{M}_q(\Z)$.
\end{enumerate}
In this case, the matrix
$\begin{pmatrix}C_{11}&D_{12}\\C_{21}&D_{22}\end{pmatrix}$
is invertible.
Moreover, for any $\te\in\mathcal T_n$ in the block form
$\te=\begin{pmatrix}\te_{11}&\te_{12}\\\te_{21}&\te_{22}\end{pmatrix}$,
where $\te_{11}$ has size $2p\times2p$, the matrix $C\te+D$ is
invertible if and only if $\te_{11}-Z$ is invertible. In this case,
\begin{equation}\label{eq:special-form-inverse}
(C\te+D)^{-1}C=\begin{pmatrix}
    (\te_{11}-Z)^{-1} & 0\\
    0 & 0
\end{pmatrix}.
\end{equation}
\end{lem}
\begin{proof}
    See \cite{Li2004}*{Lemma~3.3}.
\end{proof}
The following lemma shows that any $g\in\rm{SO}(n,n|\Z)$ can be brought to the above special form by an integral change of coordinates.
\begin{lem}\label{lem:special-form-reduction}
    Let $g=\begin{pmatrix}
        A & B\\
        C & D\\
    \end{pmatrix}\in\rm{SO}(n,n|\Z)$. Then there is an $X\in\rm{GL}_n(\Z)$ such that $g\rho(X)$ satisfies condition \textup{(1)} of Lemma~\ref{special from g} for some $p$.
\end{lem}
\begin{proof}
    See~\cite{Li2004}*{Lemma~3.5}.
\end{proof}
For the
remainder of this subsection, suppose that $g$ already
satisfies the special form of Lemma~\ref{special from g},
and put $q=n-2p$.

We now recall the explicit embedding maps
in Li's Proposition~4.1; these are the formulas on which
the later cyclic group action is constructed. Every matrix $g$ in the special form of Lemma~\ref{special from g} uniquely determines a rational skew-symmetric matrix $Z$. We can choose $m\geq 1$ such that $mZ\in\rm{M}_{2p}(\Z)$. The elementary divisor theorem for integral alternating forms gives an $R_Z\in\rm{GL}_{2p}(\Z)$ and nonzero integers $h_1,h_2,\cdots,h_r$ $(r\leq p)$ such that
\begin{equation}\label{eq:Z-normal}
mZ=R_Z^t\begin{pmatrix}
    0 & H_r & 0\\
    -H_r & 0 & 0\\
    0 & 0 & 0
\end{pmatrix} R_Z, \qquad H_r=\diag(h_1,h_2,\cdots,h_r).
\end{equation}
This is exactly the normal form used in the proof of \cite{Li2004}*{Proposition~4.1}. Write each $h_j/m$ in lowest terms as
\[
\cfrac{h_j}{m}=\cfrac{m_j}{n_j}, \quad \gcd(m_j,n_j)=1, \quad n_j>0,
\] and put
\[
F_{\rm{fin}}=\Z_{n_1}\times\Z_{n_2}\times\cdots\times\Z_{n_r}, \quad M=\R^p\times\Z^q\times F_{\rm{fin}}.
\]
If $r=0$, all finite blocks are understood as absent. We identify
$\widehat{F_{\rm fin}}$ with $F_{\rm fin}$ using the pairing
\[
 \langle a,b\rangle
 =\exp\!\left(2\pi i\sum_{j=1}^r\frac{a_jb_j}{n_j}\right).
\]
Thus
\[
G=M\times\widehat{M}\cong \R^p\times{\R^*}^p\times\Z^q\times\T^q\times F_{\rm{fin}}\times F_{\rm{fin}}.
\]
Define
\[
H^*=\R^p\times{\R^*}^p\times\R^q\times{\R^*}^q\times\R^r\times{\R^*}^r.
\]
The real dual coordinate spaces are identified with Euclidean
spaces through the standard characters. Let
\[
 \pi_G:\R^p\times\R^{*p}\times\Z^q\times\R^{*q}
 \times\Z^r\times\Z^r\longrightarrow G
\]
be the quotient map that is the identity on the first three
blocks, reduces the fourth block modulo $\Z^q$, and reduces
the coordinates in each of the last two blocks modulo the
corresponding $n_j$. We use the matrices
\[
 J_0=\begin{pmatrix}
 0 & I_p\\
 -I_p & 0\\
 \end{pmatrix},
 \quad
 J_1=\begin{pmatrix}
 J_0 &0 & 0\\
 0 & 0 & I_q\\
 0 & -I_q & 0\\
 \end{pmatrix},
\]
\[
 \Ld=\diag(n_1^{-1},\ldots,n_r^{-1}),
 \qquad
 J_2=\begin{pmatrix}
 0 & \Ld\\
 -\Ld & 0\end{pmatrix},
 \qquad
 J=\begin{pmatrix}J_1&0\\0&J_2\end{pmatrix}.
\]
If $z,w\in G$ are represented by elements of the indicated subgroup of $H^*$, then
\[
 \rho_G(z,w)=\exp(2\pi i z^tJw).
\]
The expression is independent of the chosen representatives,
since the periods pair integrally with this subgroup under $J$.

\begin{dfn}\label{def:embedding-map}
Let $\te\in\mathcal T_n$. An \emph{embedding map} for $\te$
is a real linear map $T:\R^n\to H^*$ satisfying the following
conditions \cite{Li2004}*{Definition~2.1}:
\begin{enumerate}
 \item $T(\Z^n)\subseteq
 \R^p\times\R^{*p}\times\Z^q\times\R^{*q}
 \times\Z^r\times\Z^r$, so $\pi_G\circ T|_{\Z^n}$
 defines a homomorphism into $G$.
 \item $\pi_G(T(\Z^n))$ is a lattice in $G$, that is,
 a discrete, cocompact subgroup.
 \item $T^tJT=\te$.
\end{enumerate}
\end{dfn}
Let $\gamma:H^*\to\R^p\times\R^{*p}\times\R^q$ be the
projection onto the first three blocks. Under condition~\textup{(1)},
condition~\textup{(2)} is equivalent to the invertibility of
\[
 \tilde T:=\gamma\circ T:
 \R^n\longrightarrow\R^p\times\R^{*p}\times\R^q.
\]
We also write $T$ for the induced map $\pi_G\circ T|_{\Z^n}$
when its values are regarded as elements of $G$.

Fix $\te\in\mathcal{T}_n$ such that $g\te$ is defined, and write it in the block form
\[
\te=\begin{pmatrix}
    \te_{11} & \te_{12}\\
    \te_{21} & \te_{22}\\
\end{pmatrix}
\] where $\te_{11}$ has size $2p\times2p$. Lemma~\ref{special from g} shows that $g\te$ is defined precisely when $\te_{11}-Z$ is invertible. This matrix is skew-symmetric, since both $\te_{11}$ and $Z$ are. Set $F_0=(\te_{11}-Z)^{-1}$. Then there is a $T_{11}\in\rm{GL}_{2p}(\R)$ such that $T_{11}^tJ_0T_{11}=\te_{11}-Z$. Put $T_{31}=\te_{21}$ and choose $T_{32}\in\rm{M}_{q}(\R)$ with $T_{32}-T_{32}^t=\te_{22}$; for instance, $T_{32}=\te_{22}/2$. Let $M_r=\diag(m_1,m_2,\cdots,m_r)$ and define
\[
E_r=\begin{pmatrix}
    M_r & 0 & 0\\
    0 & I_r & 0\\
\end{pmatrix}\in \rm{M}_{2r\times 2p}(\R),
\] set
\[
T_1=\begin{pmatrix}
    T_{11} & 0\\
    0 & I_q\\
    T_{31} & T_{32}
\end{pmatrix},\quad T_2=\begin{pmatrix}
    E_rR_Z & 0_{2r\times q}
\end{pmatrix},
\quad T=\begin{pmatrix}
    T_1\\
    T_2\\
\end{pmatrix}.
\]
This first embedding realizes the original parameter $\te$.
The map $T:\R^n\to H^*$ is an embedding map for $\te$
in the sense of Definition~\ref{def:embedding-map}. Indeed,
its restriction to $\Z^n$ takes values in the subgroup on
which the quotient map to $G$ is defined, and a direct
block computation gives $T^tJT=\te$. The map $\tilde{T}:=\gamma\circ T:\R^n\to
\R^p\times{\R^*}^p\times\R^q$ has matrix
\[
\begin{pmatrix}
    T_{11} & 0\\
    0 & I_q\\
\end{pmatrix},
\]
which is invertible. Consequently, the image
$D:=T(\Z^n)$ in $G$ is a lattice. We will construct a second
embedding map $S$ such that $S(\Z^n)=D^\perp$, where the
\emph{symplectic dual} of $D=T(\Z^n)$ is the subgroup
\[
D^\perp:=\{y\in G:\rho_G(y,d)=1\text{ for every }d\in D\}.
\]
Note that
$\gcd(m_j,n_j)=1$, so there are integers $c_j,d_j$ such that $c_jm_j+d_jn_j=1$. Set
\[
Q_2=\diag(c_1,c_2,\ldots,c_r),\qquad
\Gm=\diag(\Ld,\Ld,-I_{2p-2r}).
\]
We now define an $(n+q+2r)\times n$ real matrix
\[
S=\begin{pmatrix}
    J_0(T_{11}^t)^{-1}R_Z^t\Gm & -J_0(T_{11}^t)^{-1}T_{31}^t\\
    0 & I_q\\
    0 & T_{32}^t\\
    \begin{pmatrix}
        0 & -I_r & 0\\
        Q_2 & 0 & 0\\
    \end{pmatrix} & 0
\end{pmatrix}.
\]
By Li~\cite{Li2004}*{Proposition~4.1}, $S$ is an embedding map for
$-\te'$, where $\te':=-S^tJS$, and $S(\Z^n)=D^\perp$ in $G$.
The map $\tilde{S}:=\gamma\circ S:\R^n\to
\R^p\times{\R^*}^p\times\R^q$ has matrix
\[
\begin{pmatrix}
    J_0(T_{11}^t)^{-1} R_Z^t\Gm & -J_0(T_{11}^t)^{-1}T_{31}^t\\
    0 & I_q\\
\end{pmatrix}.
\]

Define
\[
\A=-\tilde{T}^{-1}\tilde{S},\qquad
\Phi=\begin{pmatrix}
    F_0 & 0\\
    0 & 0
\end{pmatrix},\qquad F_0=(\te_{11}-Z)^{-1}.
\]
Following \cite{Schwarz1998}*{Equation~53}, put
\[
C'=\A^{-1}\Phi,\quad D'=\A^{-1}-C'\te,
\]
\[
A'=\A^t+\te'C',\quad B'=\te'\A^{-1}-A'\te.
\]
By Li~\cite{Li2004}*{Proposition~4.1 and Remark~4.3(2)},
\[
g'=\begin{pmatrix}
    A' & B'\\
    C' & D'
\end{pmatrix}\in\rm{SO}(n,n|\Z),\qquad g'\te=\te',
\] and there exist $R_0\in\rm{GL}_n(\Z)$ and
$K\in\mathcal T_n\cap\rm{M}_n(\Z)$ such that
\[
g=\rho(R_0)\nu(K)g'.
\]
The matrices $g', K, R_0$ depend only on $g$ and its special-form data, not on the particular $\te$.

Before describing the bimodule structure, we specify the two noncommutative
torus algebras and their smooth subalgebras. In the Weyl convention
fixed in the introduction, we use the twisted group $\mathrm{C}^*$-algebra
realizations
\[
 \mathcal A:=A_\te\cong C^*(\Z^n,\sigma_\te),\qquad
 \mathcal B:=A_{\te'}\cong C^*(\Z^n,\sigma_{\te'}).
\]
Here $\te'=g'\te$ is the parameter determined by the embedding
map $S$. Denote by $\mathcal A^\infty:=A_\te^\infty=\mathcal S(\Z^n,\sigma_\te)$
and $\mathcal B^\infty:=A_{\te'}^\infty=\mathcal S(\Z^n,\sigma_{\te'})$
the dense smooth subalgebras of $\mathcal A$ and $\mathcal B$,
respectively, consisting of Fourier series with rapidly decreasing coefficients.

We now construct the explicit equivalence bimodule on
$E^\infty:=\mathcal S(M)$.
Li~\cite{Li2004}*{Proposition~2.2}, based on Rieffel's construction~\cite{Rieffel1988}*{Theorem~2.15 and Corollary~3.8}, states that if $T$ and $S$ are embedding maps for $\te$ and $-\te'$ with
\[
S(\Z^n)=T(\Z^n)^\perp
\]
then $E^\infty$ is a complete Morita equivalence
$\mathcal B^\infty$--$\mathcal A^\infty$ bimodule, whose
completion in the right Hilbert-module norm specified below is a $\mathcal B$--$\mathcal A$ imprimitivity bimodule.
The explicit formulas for the module actions and inner products are given below.

For $x\in\Z^n$, write
\[
\pi_G(T(x))=(T'(x),T''(x)),\qquad
\pi_G(S(x))=(S'(x),S''(x))\in M\times\widehat{M}.
\]
Let $J'$ be the strictly upper-triangular
part of $J$, so that $J=J'-(J')^t$. For $x\in\Z^n$, define
operators on $\mathcal{S}(M)$ by
\begin{equation}\label{eq:R-L-operators}
 \begin{aligned}
 (R_xf)(m)
 &=\exp\!\left(-\pi i\,T(x)^tJ'T(x)\right)
   \langle{m},{T''(x)}\rangle f(m-T'(x)),\\
 (L_xf)(m)
 &=\exp\!\left(-\pi i\,S(x)^tJ'S(x)\right)
   \langle {m}, {-S''(x)}\rangle f(m+S'(x)).
 \end{aligned}
\end{equation}
These operators define the right $\mathcal A^\infty$-module action and the left $\mathcal B^\infty$-module action on $E^\infty$ by
\[
f\cdot U_x^{\te}=R_xf,\qquad U_x^{\te'}\cdot f=L_xf.
\]
In particular,
\[
R_x R_y=\overline{\sm_{\te}(x,y)}R_{x+y},\quad L_xL_y=\sm_{\te'}(x,y)L_{x+y}.
\]
Define the phase factors
\[
c_T(x):=\exp\!\left(-\pi i\,T(x)^tJ'T(x)\right),
\qquad
c_S(x):=\exp\!\left(-\pi i\,S(x)^tJ'S(x)\right).
\]
Then, in terms of the Heisenberg--Weyl operators,
\[
R_x=c_T(x)\pi(Tx),\qquad L_x=c_S(x)\pi(-Sx).
\]

Fix Haar measure $dm$ on $M$ and use the convention
\[
 (f,h)_2=\int_M\overline{f(m)}h(m)\,dm
\]
for the Hilbert-space inner product, which is linear in its
second variable. For $f,h\in\mathcal S(M)$, the right
$A_{\te}^\infty$-valued inner product has Fourier coefficients
\begin{equation}\label{eq:right ip}
 \langle f,h\rangle_{A_{\te}^\infty}(x)
 =\exp\!\left(-\pi i\,T(x)^tJ'T(x)\right)
  \int_M\langle m,-T''(x)\rangle\,
  h(m+T'(x))\,\overline{f(m)}\,dm
 =(f,R_{-x}h)_2.
\end{equation}
The left $A_{\te'}^\infty$-valued inner product has Fourier coefficients
\begin{equation}\label{eq:left ip}
 {}_{A_{\te'}^\infty}\langle f,h\rangle(x)
 =\kappa\exp\!\left(\pi i\,S(x)^tJ'S(x)\right)
  \int_M\langle m,S''(x)\rangle\,
  f(m)\,\overline{h(m+S'(x))}\,dm
 =\kappa(L_xh,f)_2.
\end{equation}
Here $\kappa>0$ is the normalization constant in Li's construction.
These module actions and inner products are those of
Li~\cite{Li2004}*{Proposition~2.2}. Completing $\mathcal S(M)$
in the norm $\|f\|=\|\langle f,f\rangle_{A_\te^\infty}\|_{A_\te}^{1/2}$
gives an $A_{\te'}$--$A_\te$ imprimitivity bimodule $E'$.
The smooth inner products extend continuously to an
$A_\te$-valued right inner product and an $A_{\te'}$-valued
left inner product on $E'$.

\section{A symplectic automorphism for the lattice action}\label{sec:symplectic-construction}

Assume the hypotheses of Theorem~\ref{thm:main}, with $g$
in the special form of Lemma~\ref{special from g}. Thus
$W\in\GL_n(\Z)$ satisfies $W^t\te W=\te$, $W^k=I_n$ and
$1\notin\rm{spec}(W)$, and
$g\rho(W^t)g^{-1}=\rho(P)$ for some $P\in\GL_n(\Z)$.
The matrix $W$ induces the $\Z_k$-action
$\af_W(U_x^\te)=U_{Wx}^\te$ on $A_\te$.

Fix the $A_{\te'}$--$A_\te$ imprimitivity bimodule $E'$ and
the embedding maps $T,S$ constructed in Section~\ref{Prelim}, and recall
that $G=M\times\widehat M$.
Our aim is to construct a symplectic automorphism $Q$ of $G$
satisfying $QT=TW$ as maps $\Z^n\to G$ and $Q^k=\Id_G$.
In the next section, we will use $Q$ to construct a compatible
cyclic group action on the bimodule.

We first transport the left module structure through the integral
shift in the factorization
\[
 g=\rho(R_0)\nu(K)g',\qquad \te'=g'\te,
\]
where $R_0\in\GL_n(\Z)$ and $K$ is an integral
skew-symmetric matrix. Put $\Th=\te'+K$ and, for
$x\in\Z^n$, set
\[
 q_K(x)=\sum_{i<j}K_{ij}x_ix_j,\qquad
 \ep_K(x)=(-1)^{q_K(x)}.
\]
The map
\begin{equation}\label{K isomorphism}
 \Psi_K:A_{\te'}\longrightarrow A_\Th,\qquad
 \Psi_K(U_x^{\te'})=\ep_K(x)U_x^\Th
\end{equation}
is a $*$-isomorphism, with inverse
$U_x^\Th\mapsto\ep_K(x)U_x^{\te'}$.
The transported left action and inner product are
\[
 b\cdot f:=\Psi_K^{-1}(b)\cdot f,\qquad
 {}_{A_\Th}\langle f,h\rangle
 :=\Psi_K\bigl({}_{A_{\te'}}\langle f,h\rangle\bigr).
\]
Together with the original right action and inner product,
these make $E'$ an $A_\Th$--$A_\te$ imprimitivity bimodule.
On $\mathcal S(M)$, the left Heisenberg--Weyl operators are
$\widetilde L_x=\ep_K(x)L_x$. The isomorphism $\Psi_K$
preserves rapid decay, so the transport also gives a smooth
$A_\Th^\infty$--$A_\te^\infty$ equivalence bimodule.

Define
\begin{equation}\label{eq:P_1 def}
 g''=\nu(K)g',\qquad P_1=R_0^{-1}PR_0.
\end{equation}
Then $g=\rho(R_0)g''$, $\Th=g''\te$, and
\begin{equation}\label{g''eq}
 g''\rho(W^t)(g'')^{-1}=\rho(P_1).
\end{equation}
In particular,
\begin{equation}\label{eq:P_1}
 P_1\Th P_1^t=\Th.
\end{equation}
Taking $k$-th powers in~\eqref{g''eq} and using the injectivity
of $\rho$ gives $P_1^k=I_n$. Together with~\eqref{eq:P_1},
this shows that $\af_{P_1^t}$ defines a $\Z_k$-action on $A_\Th$.

The conjugacy relation also connects the two lattice actions
through the embedding maps and constrains the block form of $W$.
\begin{lem}\label{W and A compatibility}
The matrix $\A=-\tilde T^{-1}\tilde S$ satisfies
\[
 W\A=\A P_1^t,\qquad W\Ph W^t=\Ph,
 \qquad \text{ where }~
 \Ph=\begin{pmatrix}F_0&0\\0&0\end{pmatrix}\text{ and } F_0=(\te_{11}-Z)^{-1}.
\]
\end{lem}

\begin{prp}\label{prop:W-block}
The matrix $W$ has the form
\[
 W=\begin{pmatrix}W_1&W_2\\0&W_4\end{pmatrix},
 \qquad W_1\in\GL_{2p}(\Z),\quad W_4\in\GL_q(\Z),
\]
and
\[
 W_1^t(\te_{11}-Z)W_1=(\te_{11}-Z),
 \qquad W_1^tZW_1=Z.
\]
\end{prp}
For both proofs, see Appendix~\ref{app:block-identities}.

We now construct the automorphism $Q$ using this block form. Write
\[
 G=G_0\times\Z^q\times\T^q,\qquad
 G_0=(\R^p\times\R^{*p})
       \times(F_{\rm fin}\times F_{\rm fin}).
\]
Let
\[
 \phi_F:\Z^{2p}\longrightarrow F_{\rm fin}\times F_{\rm fin},
 \qquad \phi_F(a)=[E_rR_Za],
\]
be the finite component of $T$, and define
$\ta:\Z^{2p}\to G_0$ by
\begin{equation}\label{eq:def-tau}
 \ta(a)=(T_{11}a,\phi_F(a)),\qquad a\in\Z^{2p}.
\end{equation}
Then
\begin{equation}\label{eq:T-tau}
 T(a,b)=\bigl(\ta(a),b,[\te_{21}a+T_{32}b]\bigr),
 \qquad (a,b)\in\Z^{2p}\oplus\Z^q,
\end{equation}
where brackets in the last coordinate denote the class in $\T^q$. Write $\rh_0$
for the restriction of $\rh_G$ to $G_0$. We have
\begin{equation}\label{eq:rho0-tau}
 \rh_0(\ta(a),\ta(a'))
 =\exp(2\pi i\,a^t\te_{11}a').
\end{equation}
Put $Q_{\R}=T_{11}W_1T_{11}^{-1}$ on the real factor.
\begin{lem}\label{lem:finite-quotient}
The homomorphism $\phi_F$ is surjective and has kernel
\[
 \Ld_Z=\{a\in\Z^{2p}:Za\in\Z^{2p}\}.
\]
Consequently,
$F_{\rm fin}\times F_{\rm fin}\cong\Z^{2p}/\Ld_Z$.
\end{lem}
Under this identification, the finite bicharacter is
\begin{equation}\label{eq:rhoF-Z}
 \rh_F(\phi_F(a),\phi_F(a'))
 =\exp(2\pi i\,a^tZa').
\end{equation}
Since $W_1^tZW_1=Z$, we have $W_1\Ld_Z=\Ld_Z$
and an induced symplectic automorphism $Q_F$ defined by
\begin{equation}\label{eq:QF-phiF}
 Q_F\phi_F(a)=\phi_F(W_1a).
\end{equation}
Thus $Q_0=Q_{\R}\oplus Q_F$ is a symplectic
automorphism of $G_0$ satisfying
\begin{equation}\label{eq:Q0-tau}
 Q_0\ta(a)=\ta(W_1a).
\end{equation}
See Appendix~\ref{app:real-finite} for the verifications.

The relation $QT=TW$ requires the first two components of
$Q(T(a,b))$ to be $\ta(W_1a+W_2b)$ and $W_4b$.
Thus the term contributed by $W_2$ must be added to the
$G_0$-coordinate. Define $c:\Z^q\to G_0$ by
$c(b)=\ta(W_2b)$. This addition produces mixed terms in the
bicharacter, which we compensate for in the torus coordinate.
Since $c$
is a homomorphism and $\rh_0$ is a bicharacter, for each
$u\in G_0$ the map $z\mapsto\rh_0(Q_0u,c(z))$ is a
character of $\Z^q$. Under the standard identification
$\widehat{\Z^q}\cong\T^q$, every such character has the form
$z\mapsto\exp(2\pi i z^tt)$ for a unique $t\in\T^q$.
Thus there is a unique map $\ell:G_0\to\T^q$ satisfying
\begin{equation}\label{eq:ell-def}
 \exp(2\pi i z^t\ell(u))=\rh_0(Q_0u,c(z)),
 \qquad u\in G_0,\ z\in\Z^q.
\end{equation}
The homomorphism property of $Q_0$ and the bicharacter
identity in the first variable show that the characters
associated to $\ell(u+v)$ and $\ell(u)+\ell(v)$ agree;
uniqueness therefore gives $\ell(u+v)=\ell(u)+\ell(v)$.
Moreover, if $e_1,\ldots,e_q$ is the standard basis of
$\Z^q$, then
\[
 \exp(2\pi i\ell_j(u))=\rh_0(Q_0u,c(e_j)),\qquad 1\leq j\leq q.
\]
Each right-hand side is continuous in $u$. Since
$t\mapsto\exp(2\pi i t)$ identifies $\R/\Z$ homeomorphically
with the unit circle, every coordinate $\ell_j$ is continuous.
Hence $\ell$ is a continuous homomorphism.
Its defining identity
cancels the mixed pairings between $Q_0u$ and $c(b)$.
The remaining term in the torus coordinate is chosen so that
$QT=TW$ also holds in that coordinate.
Put
\[
 \Xi=W_4^t\te_{21}W_2+W_4^tT_{32}W_4-T_{32}
 \in\rm{M}_q(\R).
\]
Define $Q:G\to G$ by
\begin{equation}\label{eq:Q-def}
 Q(u,b,t)=\left(Q_0u+c(b),\,W_4b,\,
 (W_4^t)^{-1}(t+\ell(u)+\Xi b)\right).
\end{equation}
Here $(u,b,t)\in G_0\times\Z^q\times\T^q$, and all
operations in the last coordinate are taken in $\T^q$.
The factor $(W_4^t)^{-1}$ is the dual action corresponding to
$b\mapsto W_4b$. The identity
$\Xi-\Xi^t=-W_2^t\te_{11}W_2$ then cancels the pairings
between $c(b)$ and $c(d)$; the calculations establishing this
identity and $QT=TW$ are given in the appendix.

\begin{prp}\label{prp:Q symplectic}
The map $Q$ in~\eqref{eq:Q-def} is a continuous group
automorphism of $G$ and preserves the symplectic bicharacter:
\[
 \rh_G(Qx,Qy)=\rh_G(x,y),\qquad x,y\in G.
\]
\end{prp}
See Appendix~\ref{app:symplecticity} for the proof.

\begin{prp}\label{prop:QT}
The symplectic automorphism $Q$ satisfies $QT=TW$ as
maps $\Z^n\longrightarrow G$.
\end{prp}
See Appendix~\ref{app:embedding} for the proof.

With the embedding maps fixed, the construction defines
$Q_V$ for every $V$ in the subgroup
\[
 \mathcal G_{\te,Z}=\left\{
 V=\begin{pmatrix}A_V&B_V\\0&H_V\end{pmatrix}
 \in\GL_n(\Z):
 V^t\te V=\te,\ A_V^tZA_V=Z\right\}.
\]
In particular, $Q=Q_W$.
\begin{lem}\label{lem:Q-functorial}
For $V_1,V_2\in\mathcal G_{\te,Z}$, one has
$Q_{V_1}Q_{V_2}=Q_{V_1V_2}$.
\end{lem}
See Appendix~\ref{app:composition} for the proof.
\begin{cor}\label{cor:Q-order}
The symplectic automorphism $Q$ has the same order as
$W$. In particular, $Q^k=\Id_G$.
\end{cor}
\begin{proof}
The construction gives $Q_{I_n}=\Id_G$, and
Lemma~\ref{lem:Q-functorial} gives $Q^j=Q_{W^j}$.
Thus $W^j=I_n$ implies $Q^j=\Id_G$. Conversely,
$Q^j=\Id_G$ and $QT=TW$ imply $TW^j=T$.
The map $T$ is injective by~\eqref{eq:T-tau}, since
$T_{11}$ is invertible, so $W^j=I_n$.
\end{proof}

\section{The equivariant imprimitivity bimodule}\label{sec:equivariant-bimodule}

In this section we construct the bimodule in Theorem~\ref{thm:main} by giving Li's
bimodule an action of $\Z_k$ compatible with the prescribed
actions on the two tori. The action must respect both module
structures and both inner products. By the crossed-product
imprimitivity theorem, proved independently by
Combes~\cite{Combes1984} and Curto--Muhly--Williams
\cite{CMW1984}*{Theorem~1}, an equivariant imprimitivity bimodule
then gives a Morita equivalence between the corresponding
crossed products. We describe the construction here and collect
the proofs in Appendix~\ref{app:symplectic-computations}.

We first assume that $g$ has the special form of
Lemma~\ref{special from g} and keep the hypotheses and notation
of Section~\ref{sec:symplectic-construction}. Thus
\[
 g=\rho(R_0)\nu(K)g',\qquad
 \te'=g'\te,\qquad \Th=\te'+K,\qquad P_1=R_0^{-1}PR_0.
\]
Recall Li's $A_{\te'}$--$A_\te$ bimodule $E'$ from Section~\ref{Prelim},
whose left action and inner product were transported through
$\Ps_K:A_{\te'}\to A_\Th$ in Section~\ref{sec:symplectic-construction}. We shall adjust $\Ps_K$
by an automorphism that multiplies each generator by a character value,
so that one operator implements both prescribed actions. We first
obtain exact covariance of the right action, then normalize the
operator to have $k$-th power equal to the identity, and finally
adjust the left action and verify covariance of the inner products.

More precisely, we seek a unitary $\mathcal U_0$ on $L^2(M)$
preserving $\mathcal S(M)$ such that
\[
 \mathcal U_0(b\xi a)
 =\af_{P_1^t}(b)\mathcal U_0(\xi)\af_W(a),
 \qquad \mathcal U_0^k=I,
\]
for $b\in A_\Th^\infty$, $a\in A_\te^\infty$ and
$\xi\in\mathcal S(M)$. Covariance of the inner products will
allow us to extend this action to the completed bimodule. We
begin with the unitary supplied by
Lemma~\ref{lem:unitary implementation} and consider its effect
on the right action.

By Proposition~\ref{prp:Q symplectic},
Proposition~\ref{prop:QT} and Corollary~\ref{cor:Q-order}, the
automorphism $Q$ satisfies
\[
 \rho_G(Qz,Qw)=\rho_G(z,w),\qquad QT=TW,\qquad Q^k=\Id_G.
\]
Lemma~\ref{lem:unitary implementation} therefore gives a unitary
$V_Q$ on $L^2(M)$ such that
\[
 V_Q\pi(z)V_Q^*=\eta_Q(z)\pi(Qz),\qquad z\in G,
\]
where $\eta_Q:G\to\T$ is continuous and
$V_Q\mathcal S(M)=\mathcal S(M)$. Since right multiplication by
$U_x^\te$ is given by $R_x=c_T(x)\pi(Tx)$, the relation $QT=TW$
gives
\[
 V_QR_xV_Q^*=\ch_R(x)R_{Wx},\qquad
 \ch_R(x)=c_T(x)\eta_Q(Tx)c_T(Wx)^{-1}.
\]
The relation $R_yR_x=\sm_\te(x,y)R_{x+y}$ and the equality
$\sm_\te(Wx,Wy)=\sm_\te(x,y)$ give, after conjugation by $V_Q$,
\[
 \ch_R(x+y)=\ch_R(x)\ch_R(y),\qquad x,y\in\Z^n.
\]
Thus $\ch_R$ is a character. To make $V_Q$ implement the right
action exactly, we shall compose it on the left with a
Heisenberg--Weyl operator $\pi(z_0)$ for some $z_0\in G$. Conjugation
by this operator must cancel $\ch_R(x)$ at $TWx$. This leads us
to define a character on $D=T(\Z^n)$ by
\[
 \ps(Ty)=\overline{\ch_R(W^{-1}y)},\qquad y\in\Z^n.
\]
The formula is well defined because $T$ is injective, and the
character is continuous because $D$ is discrete. As $D$ is
closed in $G$, Lemma~\ref{lem:character-extension} allows us to
extend $\ps$ to $G$. We use the same notation for a chosen
extension. By Lemma~\ref{lem:rhoG-self-duality}, there is a unique
$z_0\in G$ such that
\[
 \ps(z)=\overline{\rh_G(z_0,z)},\qquad z\in G.
\]
In particular,
$\overline{\rh_G(z_0,TWx)}=\overline{\ch_R(x)}$ for every
$x\in\Z^n$. Put $\mathcal U=\pi(z_0)V_Q$. Both factors preserve
$\mathcal S(M)$, so $\mathcal U\mathcal S(M)=\mathcal S(M)$.

\begin{prp}\label{prop:right-exact}
For every $x\in\Z^n$,
\begin{equation}\label{U covariance}
 \mathcal U R_x\mathcal U^*=R_{Wx}.
\end{equation}
Equivalently,
\[
 \mathcal U(\xi\cdot a)
 =(\mathcal U\xi)\cdot\af_W(a),
 \qquad \xi\in\mathcal S(M),\quad a\in A_\te^\infty.
\]
\end{prp}

See Appendix~\ref{app:right-exact} for the proof.

To obtain an action of $\Z_k$ on the bimodule $\mathcal{S}(M)$, we need the $k$-th power of the
implementing operator to be the identity. It will be enough to
show that $\mathcal U^k=\ld I$ for some $\ld\in\T$.
Once this is established, choose $\tilde\ld\in\T$ with $\tilde\ld^k=\ld^{-1}$
and rescale $\mathcal U$ by setting $\mathcal U_0=\tilde\ld\mathcal U$.
Then $\mathcal U_0^k=\tilde\ld^k\mathcal U^k=I$.
Since $|\tilde\ld|=1$, this rescaling preserves unitarity and
leaves conjugation unchanged:
$\mathcal U_0R_x\mathcal U_0^*=\mathcal U R_x\mathcal U^*$.
To prove that $\mathcal U^k$ is scalar, first note that
equation~\eqref{U covariance} and $W^k=I_n$ show that $\mathcal U^k$ commutes with
every $R_x$.
This only gives commutation with the operators
associated to $D=T(\Z^n)$, whereas irreducibility applies to
the full Heisenberg--Weyl representation of $G$. We therefore
first identify the possible extra term in $\mathcal U^k$ and
then use the dual lattice $D^\perp=S(\Z^n)$ to remove it.
For this purpose, we determine how $Q$ acts on $D^\perp$.

Since $Q(D)=D$ and $Q$ preserves $\rho_G$, it preserves
$D^\perp$ as well. In fact, if $s\in D^\perp$ and $d=Qd'\in D$,
then $\rho_G(Qs,d)=\rho_G(s,d')=1$; applying the same argument to
$Q^{-1}$ gives equality. Since $S$ identifies $\Z^n$ with
$D^\perp$, there is an $L\in\GL_n(\Z)$ such that $QS=SL$.
To find $L$, recall the projections of the embedding maps from
Section~\ref{Prelim}.

The projection $\gm$ discards the torus and finite coordinates,
so it is well defined on $G$, with $\tilde T=\gm T$ and
$\tilde S=\gm S$. Formula~\eqref{eq:Q-def} gives
\[
 \widetilde Q=
 \begin{pmatrix}
 T_{11}W_1T_{11}^{-1}&T_{11}W_2\\
 0&W_4
 \end{pmatrix}
 =\tilde T W\tilde T^{-1},
 \qquad \gm Q=\widetilde Q\gm.
\]
Projecting $QS=SL$ and using $\tilde S=-\tilde T\A$, we obtain
\[
 \tilde T W\A=\tilde T\A L.
\]
On the other hand, Lemma~\ref{W and A compatibility} gives
$W\A=\A P_1^t$. Since $\tilde T$ and $\A$ are invertible,
$L=P_1^t$. Consequently,
\begin{equation}\label{eq: QS}
 QS=SP_1^t.
\end{equation}

The operator $\mathcal U=\pi(z_0)V_Q$ still implements $Q$ up to
a continuous function $\eta_{\mathcal U}:G\to\T$:
\[
 \mathcal U\pi(z)\mathcal U^*
 =\eta_{\mathcal U}(z)\pi(Qz),\qquad z\in G.
\]
Explicitly, $\eta_{\mathcal U}(z)=\eta_Q(z)\overline{\rh_G(z_0,Qz)}$.
Iterating this identity and using $Q^k=\Id_G$, we get
\begin{equation}\label{eq:Uk-Weyl-covariance}
 \mathcal U^k\pi(z)\mathcal U^{-k}=\ch_k(z)\pi(z),
 \qquad
 \ch_k(z)=\prod_{j=0}^{k-1}\eta_{\mathcal U}(Q^jz).
\end{equation}
The following lemma shows that the extra term in $\mathcal U^k$
comes from a point of $D^\perp$ fixed by $Q$.

\begin{lem}\label{lem:Uk-obstruction}
There exist $\ld\in\T$ and $v\in D^\perp$ such that
\[
 \mathcal U^k=\ld\pi(v),\qquad Qv=v.
\]
\end{lem}

See Appendix~\ref{app:Uk-obstruction} for the proof.

The hypothesis $1\notin\operatorname{spec}(W)$ now enters the
argument. 
Write the point in the preceding lemma as $v=S(m)$,
where $m\in\Z^n$. By~\eqref{eq: QS}, the equality $Qv=v$ gives
$S(P_1^tm)=S(m)$, and injectivity of $S$ gives $P_1^tm=m$.
Since $P_1^t=\A^{-1}W\A$, the matrix $P_1^t-I_n$ is
invertible. Thus $m=0$, so $v=0$ and $\mathcal U^k=\ld I$.
We may therefore apply the scalar normalization described above.

\begin{prp}\label{prp:modified U}
There exists $\tilde\ld\in\T$ such that
$\mathcal U_0=\tilde\ld\mathcal U$ is unitary on $L^2(M)$,
preserves $\mathcal S(M)$, and satisfies
\[
 \mathcal U_0^k=I,\qquad
 \mathcal U_0R_x\mathcal U_0^*=R_{Wx},\qquad x\in\Z^n.
\]
\end{prp}

See Appendix~\ref{app:modified-U} for the proof.

Fix $\mathcal U_0$ as in Proposition~\ref{prp:modified U}.
Since it is a scalar multiple of $\mathcal U$, it also satisfies
$\mathcal U_0\pi(z)\mathcal U_0^*=\eta_{\mathcal U}(z)\pi(Qz)$.
We have now obtained the required cyclic action on the right
$A_\te$-module. To make the bimodule equivariant, the same
operator must also agree with the action of $P_1^t$ on
$A_\Th$. We shall achieve this by adjusting $\Ps_K$, keeping
$\mathcal U_0$ fixed. Recall that $L_x$ denotes the original
left action of $U_x^{\te'}$, before transport to $A_\Th$.

Since $L_x=c_S(x)\pi(-Sx)$, the covariance of $\mathcal U_0$
and~\eqref{eq: QS} give scalars $\ch_L(x)\in\T$ such that
\begin{equation}\label{eq:U_0 and L_x}
 \mathcal U_0L_x\mathcal U_0^*=\ch_L(x)L_{P_1^tx}.
\end{equation}
The calculation that made $\ch_R$ a character does not apply
directly to $\ch_L$, because the original left parameter is
$\te'$. Conjugating
$L_xL_y=\sm_{\te'}(x,y)L_{x+y}$ gives
\[
 \frac{\ch_L(x)\ch_L(y)}{\ch_L(x+y)}
 =\frac{\sm_{\te'}(x,y)}{\sm_{\te'}(P_1^tx,P_1^ty)}.
\]
Recall that $P_1^t$ preserves $\Th=\te'+K$. Substituting this
expression into $P_1\Th P_1^t=\Th$ gives
\[
 P_1\te'P_1^t-\te'=K-P_1KP_1^t.
\]
Thus $P_1^t$ need not preserve the multiplier for $\te'$.
We instead express the left action using the Weyl generators
of $A_\Th$, whose multiplier is preserved by $P_1^t$.

Recall from~\eqref{K isomorphism} that
$\Ps_K(U_x^{\te'})=\ep_K(x)U_x^\Th$, where
$\ep_K(x)\in\{1,-1\}$. On the transported module, left
multiplication by $U_x^\Th$ is therefore represented by
\[
 \widetilde L_x=\ep_K(x)L_x.
\]
These operators satisfy
$\widetilde L_x\widetilde L_y=\sm_\Th(x,y)\widetilde L_{x+y}$.
Equation~\eqref{eq:U_0 and L_x} becomes
\begin{equation}\label{eq:U_0 and L tilde}
 \begin{aligned}
 \mathcal U_0\widetilde L_x\mathcal U_0^*
 &=\dt(x)\widetilde L_{P_1^tx},\\
 \dt(x)&=\frac{\ep_K(x)}{\ep_K(P_1^tx)}\ch_L(x).
 \end{aligned}
\end{equation}
Preservation of $\sm_\Th$ gives the following conclusion.

\begin{lem}\label{lem:delta-character}
The function $\dt:\Z^n\to\T$ defined
in~\eqref{eq:U_0 and L tilde} is a character.
\end{lem}

See Appendix~\ref{app:delta-character} for the proof.

It remains to remove $\dt$. We do this by changing the scalar
factors in the identification of the left Weyl generators. For
a character $\om:\Z^n\to\T$, define the automorphism $\gm_\om$ of $A_\Th$ by
\[
 \gm_\om(U_x^\Th)=\om(x)U_x^\Th.
\]
Because $\om$ is a character, this formula preserves
multiplication and adjoints. Compose it with $\Ps_K$ and put
\begin{equation}\label{eq:Psi-K-omega}
 \begin{aligned}
 \Ps_{K,\om}&=\gm_\om\circ\Ps_K:A_{\te'}\longrightarrow A_\Th,\\
 \Ps_{K,\om}(U_x^{\te'})
 &=\om(x)\overline{\ep_K(x)}U_x^\Th.
 \end{aligned}
\end{equation}
Here $\overline{\ep_K(x)}=\ep_K(x)$ because $\ep_K(x)\in\{1,-1\}$,
so the displayed formula indeed represents $\gm_\om\circ\Ps_K$.
The maps $\Ps_K$ and $\gm_\om$ multiply Fourier coefficients by
scalars of modulus one, so they preserve rapid decrease.
Transporting the left action and inner product through
$\Ps_{K,\om}$ therefore gives a smooth
$A_\Th^\infty$--$A_\te^\infty$ equivalence bimodule on
$\mathcal S(M)$. Its completion is an $A_\Th$--$A_\te$
imprimitivity bimodule, denoted by ${}_{A_\Th}E_{A_\te}$.
The right action and right inner product remain those of $E'$.

Since
$\Ps_{K,\om}^{-1}(U_x^\Th)
=\ep_K(x)\overline{\om(x)}U_x^{\te'}$, the new left Heisenberg--Weyl
operators are
\begin{equation}\label{eq:L-omega}
 L_x^{(\om)}=\ep_K(x)\overline{\om(x)}L_x.
\end{equation}
Using~\eqref{eq:U_0 and L tilde}, their covariance relation is
\[
 \mathcal U_0L_x^{(\om)}\mathcal U_0^*
 =\dt(x)\frac{\om(P_1^tx)}{\om(x)}L_{P_1^tx}^{(\om)}.
\]
Thus the factor multiplying $L_{P_1^tx}^{(\om)}$ will be one
provided that
\begin{equation}\label{eq:gauge-character-equation}
 \frac{\om(P_1^tx)}{\om(x)}=\dt(x)^{-1},\qquad x\in\Z^n.
\end{equation}
We can solve this equation because $P_1-I_n$ is invertible. This
is the second use of the hypothesis that $W$ has no eigenvalue
$1$.

\begin{lem}\label{lem:gauge-solution}
There is a character $\om:\Z^n\to\T$
satisfying~\eqref{eq:gauge-character-equation}.
\end{lem}

See Appendix~\ref{app:character-solution} for the proof.

Fix a character $\om$ as in the lemma and use $\Ps_{K,\om}$ for
the left identification. Since the operator $\mathcal U_0$ has
not changed, its right covariance and the relation
$\mathcal U_0^k=I$ still hold. We have also obtained the desired
left covariance.

\begin{prp}\label{prop:left-exact}
For the left action transported through $\Ps_{K,\om}$,
\begin{equation}\label{eq:left-exact}
 \mathcal U_0(b\xi)=\af_{P_1^t}(b)\mathcal U_0\xi,
 \qquad b\in A_\Th^\infty,\quad \xi\in\mathcal S(M).
\end{equation}
\end{prp}

See Appendix~\ref{app:left-exact} for the proof.

Combining Propositions~\ref{prop:right-exact},
\ref{prp:modified U} and~\ref{prop:left-exact}, we have obtained
\begin{equation}\label{eq:two-sided-covariance-core}
 \mathcal U_0(b\xi a)
 =\af_{P_1^t}(b)\mathcal U_0(\xi)\af_W(a),
 \qquad \mathcal U_0^k=I,
\end{equation}
for $b\in A_\Th^\infty$, $a\in A_\te^\infty$ and
$\xi\in\mathcal S(M)$. It remains to check the two inner
products. Besides giving the required equivariance, the right
inner-product identity will show that $\mathcal U_0$ is an
isometry for the module norm and hence extends to $E$.

For this calculation, we express the Fourier coefficients of
the inner products using the $L^2$ pairing. This will allow us
to use the unitarity of $\mathcal U_0$ on $L^2(M)$.

Recall that Li's original left inner product has expansion
\[
 {}_{A_{\te'}}\langle f,h\rangle
 =\sum_{x\in\Z^n}c_x(f,h)U_x^{\te'},
\]
where, by~\eqref{eq:left ip},
\begin{equation}\label{eq:left-ip-theta-prime}
 c_x(f,h)=\kappa(L_xh,f)_2.
\end{equation}
Here $\kappa>0$ is Li's normalization constant and
$(\cdot,\cdot)_2$ is linear in its second variable. The
transported left inner product is
\begin{equation}\label{eq:left-ip-transport}
 {}_{A_\Th}\langle f,h\rangle
 =\Ps_{K,\om}\bigl({}_{A_{\te'}}\langle f,h\rangle\bigr).
\end{equation}
Its Fourier coefficient at $x$ is
$\kappa\overline{\ep_K(x)}\om(x)(L_xh,f)_2$. Since the
Hilbert-space inner product is conjugate-linear in its first
variable,~\eqref{eq:L-omega} gives
\[
 (L_x^{(\om)}h,f)_2
 =\overline{\ep_K(x)}\om(x)(L_xh,f)_2.
\]
Consequently,
\begin{equation}\label{eq:left-ip-coefficient}
 {}_{A_\Th}\langle f,h\rangle(x)
 =\kappa(L_x^{(\om)}h,f)_2,\qquad x\in\Z^n.
\end{equation}
The right inner product is unchanged, and~\eqref{eq:right ip}
reads
\begin{equation}\label{eq:right-ip-coefficient}
 \langle f,h\rangle_{A_\te}(x)=(f,R_{-x}h)_2,
 \qquad x\in\Z^n.
\end{equation}

\begin{prp}\label{prop:inner-product-covariance}
For every $f,h\in\mathcal S(M)$,
\begin{align}
 \langle\mathcal U_0f,\mathcal U_0h\rangle_{A_\te}
 &=\af_W\bigl(\langle f,h\rangle_{A_\te}\bigr),
 \label{eq:right-ip-covariance}\\
 {}_{A_\Th}\langle\mathcal U_0f,\mathcal U_0h\rangle
 &=\af_{P_1^t}\bigl({}_{A_\Th}\langle f,h\rangle\bigr).
 \label{eq:left-ip-covariance}
\end{align}
The operator $\mathcal U_0$ extends uniquely to a surjective
isometry of ${}_{A_\Th}E_{A_\te}$. Both inner-product identities
and the module covariance~\eqref{eq:two-sided-covariance-core}
hold on this completion, where $\mathcal U_0^k=I$ as well.
\end{prp}

See Appendix~\ref{app:inner-product-covariance} for the proof.

We have now checked all the conditions for $E$ to be a
$\Z_k$-equivariant $A_\Th$--$A_\te$ imprimitivity bimodule,
where the action on $E$ is generated by $\mathcal U_0$. The
crossed-product imprimitivity theorem therefore gives the
following result.

\begin{prp}\label{prp:Th ME}
For $\Th=\te'+K$ and $P_1=R_0^{-1}PR_0$ as in
\eqref{eq:P_1 def},
\[
 A_\Th\rtimes_{\af_{P_1^t}}\Z_k
 \sim_M A_\te\rtimes_{\af_W}\Z_k.
\]
\end{prp}

\begin{proof}
Apply the crossed-product imprimitivity theorem
\cites{Combes1984,CMW1984} to $E$ with the action generated by
$\mathcal U_0$, using Proposition~\ref{prop:inner-product-covariance}
and~\eqref{eq:two-sided-covariance-core}.
\end{proof}

To finish the proof of the main theorem, we must return from
the parameter $\Th$ to $g\te$ and remove the special-form
assumption on $g$. Both steps are changes of coordinates.

\begin{proof}[Proof of Theorem~\ref{thm:main}]
First suppose that $g$ has the special form used above. Recall
that $g=\rho(R_0)\nu(K)g'$, $\Th=\nu(K)g'\te$ and
$P_1=R_0^{-1}PR_0$. Proposition~\ref{prp:Th ME} gives
\begin{equation}\label{eq:normmorita}
 A_\Th\rtimes_{\af_{P_1^t}}\Z_k
 \sim_M A_\te\rtimes_{\af_W}\Z_k.
\end{equation}
Since $g\te=R_0\Th R_0^t$, define
\[
 \Ph_{R_0}:A_\Th\longrightarrow A_{g\te},\qquad
 \Ph_{R_0}(U_x^\Th)=U_{(R_0^t)^{-1}x}^{g\te}.
\]
The identity $\sm_{g\te}((R_0^t)^{-1}x,(R_0^t)^{-1}y)=\sm_\Th(x,y)$
shows that $\Ph_{R_0}$ is a $*$-isomorphism, with inverse
$U_y^{g\te}\mapsto U_{R_0^ty}^{\Th}$. Since
$(R_0^t)^{-1}P_1^t=P^t(R_0^t)^{-1}$, it is equivariant:
\[
 \Ph_{R_0}\circ\af_{P_1^t}=\af_{P^t}\circ\Ph_{R_0}.
\]
Transporting the left action and inner product by
\[
 b\cdot\xi=\Ph_{R_0}^{-1}(b)\cdot\xi,\qquad
 {}_{A_{g\te}}\langle\xi,\eta\rangle
 =\Ph_{R_0}\bigl({}_{A_\Th}\langle\xi,\eta\rangle\bigr),
\]
gives a $\Z_k$-equivariant $A_{g\te}$--$A_\te$
imprimitivity bimodule. In particular,
\begin{equation}\label{eq:special g ME}
 \begin{aligned}
 A_{g\te}\rtimes_{\af_{P^t}}\Z_k
 &\cong A_\Th\rtimes_{\af_{P_1^t}}\Z_k\sim_M A_\te\rtimes_{\af_W}\Z_k.
 \end{aligned}
\end{equation}

For general $g=\begin{psmallmatrix}A&B\\C&D\end{psmallmatrix}$,
Lemma~\ref{lem:special-form-reduction} gives $R\in\GL_n(\Z)$ such that
 $g_0:=g\rho(R)$
has special form. The blocks of $g_0$ are
\begin{equation}\label{eq:g0-blocks}
 g_0=\begin{pmatrix} 
 AR & B (R^t)^{-1}\\
 CR & D (R^t)^{-1}\end{pmatrix}.
\end{equation}
Put
\begin{equation}\label{eq:theta0-W0}
 \te_0=R^{-1}\te(R^t)^{-1},\qquad W_0=R^tW(R^t)^{-1}.
\end{equation}
Then $\te_0\in\mathcal T_n$ and $W_0\in\GL_n(\Z)$. Writing
$g_0=\begin{psmallmatrix}A_0&B_0\\C_0&D_0\end{psmallmatrix}$, we have
\begin{align}
 C_0\te_0+D_0
 &=(C\te+D)(R^t)^{-1},\label{eq:g0-denominator}\\
 A_0\te_0+B_0
 &=(A\te+B)(R^t)^{-1}.\label{eq:g0-numerator}
\end{align}
The first matrix is invertible, so $g_0\te_0$ is defined, and
\begin{equation}\label{eq:g0theta0-gtheta}
 \begin{aligned}
 g_0\te_0
 &=(A\te+B)(R^t)^{-1}\bigl((C\te+D)(R^t)^{-1}\bigr)^{-1}=g\te.
 \end{aligned}
\end{equation}
By~\eqref{eq:theta0-W0} and the hypotheses on $W$,
\begin{equation}\label{eq:W0-hypotheses}
 W_0^t\te_0W_0=\te_0,\qquad W_0^k=I_n,\qquad
 1\notin\operatorname{spec}(W_0).
\end{equation}
Since $\rho$ is a group homomorphism,
\begin{equation}\label{eq:rho-W0}
 \rho(W_0^t)=\rho(R)^{-1}\rho(W^t)\rho(R),
\end{equation}
and hence
\begin{equation}\label{eq:g0-conjugacy}
 g_0\rho(W_0^t)g_0^{-1}
 =g\rho(W^t)g^{-1}=\rho(P).
\end{equation}
Taking $k$th powers in~\eqref{eq:g0-conjugacy} and using
injectivity of $\rho$ gives $P^k=I_n$. Applying the same
identity to $g_0\te_0$, using $\rho(W_0^t)\te_0=\te_0$, gives
\[
 P(g\te)P^t=P(g_0\te_0)P^t=g_0\te_0=g\te.
\]

The special-form case therefore gives a $\Z_k$-equivariant
$A_{g\te}$--$A_{\te_0}$ imprimitivity bimodule $E_0$ and
\begin{equation}\label{eq:special-result-g0}
 A_{g_0\te_0}\rtimes_{\af_{P^t}}\Z_k
 \sim_M A_{\te_0}\rtimes_{\af_{W_0}}\Z_k.
\end{equation}
Define
\begin{equation}\label{eq:Phi-R-general}
 \Ph_R:A_{\te_0}\longrightarrow A_\te,\qquad
 \Ph_R(U_x^{\te_0})=U_{(R^t)^{-1}x}^\te.
\end{equation}
The identity $((R^t)^{-1}x)^t\te((R^t)^{-1}y)=x^t\te_0y$
shows that $\Ph_R$ is a $*$-isomorphism, with inverse
$U_y^\te\mapsto U_{R^ty}^{\te_0}$. Since
\begin{equation}\label{eq:R-W-intertwining}
 (R^t)^{-1}W_0=W(R^t)^{-1},
\end{equation}
it is equivariant:
\[
 \Ph_R\circ\af_{W_0}=\af_W\circ\Ph_R.
\]
Transport the right action and inner product of $E_0$ by
\[
 \begin{aligned}
 \xi\cdot a&=\xi\cdot_{E_0}\Ph_R^{-1}(a),\\
 \langle\xi,\eta\rangle_{A_\te}
 &=\Ph_R\bigl(\langle\xi,\eta\rangle_{A_{\te_0}}\bigr).
 \end{aligned}
\]
Equivariance of $\Ph_R$ makes the resulting bimodule $E_g$ a
$\Z_k$-equivariant $A_{g\te}$--$A_\te$ imprimitivity bimodule
with cyclic operator $\mathcal U:=\mathcal U_0$.
The crossed-product imprimitivity theorem
\cites{Combes1984,CMW1984} gives
\[
 A_{g\te}\rtimes_{\af_{P^t}}\Z_k
 \sim_M A_\te\rtimes_{\af_W}\Z_k.
\]
\end{proof}

\begin{rmk}\label{rmk:earlier-equivariant-modules}
The construction includes familiar equivariant Heisenberg right
modules. For example, let $n=2p$, let $\te$ be invertible and
take the full-duality transformation
$g=\begin{psmallmatrix}0&I_n\\I_n&0\end{psmallmatrix}$.
Then $q=0$, the finite factor is absent, and
$Q=T_{11}WT_{11}^{-1}$. Up to the phase choices above, its
unitary implementation gives the metaplectic covariance used by
Chakraborty and Luef~\cite{CL}*{Theorem~5.4} to extend the Bott
Heisenberg right module to a crossed product. Related fundamental
right modules on $\mathcal S(\R^p\times\Z^q)$ are treated by
Chakraborty~\cite{ChTrace}*{Theorem~4.2}, who also notes the
possibility of lifting the associated Morita equivalence in
\cite{ChTrace}*{Remark~4.3}. For irrational rotation algebras,
our equivariant Morita equivalence theorem, Theorem~\ref{thm:main},
recovers the equivariant Morita equivalences used in~\cite{BCHL}
to classify crossed products by finite cyclic subgroups of
$\rm{SL}_2(\Z)$ up to Morita equivalence.
Under the hypotheses of
Theorem~\ref{thm:main}, the verification of the left action and
both inner products gives such an equivariant
imprimitivity bimodule.
\end{rmk}

\section{Applications}

We now give several consequences of Theorem~\ref{thm:main}.
We begin with finite cyclic group actions on two-dimensional
noncommutative tori, where the result provides the complete Morita
classification of the corresponding crossed products. We then turn
to the canonical flip. Since $\rh(-I_n)$ is central in
$\rm{SO}(n,n|\Z)$, the conjugacy condition in
Theorem~\ref{thm:main} is automatic for the flip. Hence
$A_{g\te}\rtimes\Z_2\sim_M A_\te\rtimes\Z_2$ whenever $g\te$ is defined, in every
dimension.

The converse problem is more subtle: does Morita equivalence of the
flip crossed products imply Morita equivalence of the underlying
noncommutative tori? For nondegenerate parameters, this is known from
\cite{ChTrace}*{Corollary~5.11}. We prove that it holds without a
simplicity assumption. Combining this descent result with the structure theory
of higher-dimensional noncommutative tori yields a complete
classification of flip crossed products in dimensions at most four.
We then use integral elements in the centralizer of
$\rh(W^t)$ to construct infinite families of parameters giving
Morita-equivalent crossed products for a fixed finite-order symmetry
$W$ satisfying $1\notin\operatorname{spec}(W)$.

\subsection{Finite cyclic group actions on rotation algebras}

We first consider finite cyclic group actions on two-dimensional
noncommutative tori. Put
\[
 J_0=\begin{pmatrix}0&1\\-1&0\end{pmatrix},\qquad
 \te_t=tJ_0,\qquad t\in\R,
\]
and write $A_t=A_{\te_t}$. Up to conjugacy, the nontrivial finite
cyclic subgroups of $\rm{SL}_2(\Z)$ have orders $2,3,4$ and $6$. We
choose the following generators:
\[
 W_2=-I_2,\qquad
 W_3=\begin{pmatrix}0&1\\-1&-1\end{pmatrix},\qquad
 W_4=\begin{pmatrix}0&1\\-1&0\end{pmatrix},\qquad
 W_6=\begin{pmatrix}1&1\\-1&0\end{pmatrix}.
\]
Thus $W_k$ has order $k$. Since $W^tJ_0W=J_0$ for every
$W\in\rm{SL}_2(\Z)$, each $W_k$ preserves $\te_t$ and hence defines
an action of $\Z_k$ on $A_t$. Let us write
\[
 C_{t,k}=A_t\rtimes_{\af_{W_k}}\Z_k,
 \qquad k\in\{2,3,4,6\}.
\]

Recall that $\rm{GL}_2(\Z)$ acts partially on $\R$ by fractional
linear transformations. More precisely, for
$M=\begin{pmatrix}a&b\\c&d\end{pmatrix}\in\rm{GL}_2(\Z)$
and $t\in\R$ such that $ct+d\neq0$, we set
\[
 M\cdot t=\frac{at+b}{ct+d}.
\]
We first show that this action, which gives the Morita equivalence
classes of rotation algebras, also preserves the Morita equivalence
classes of the corresponding crossed products.

\begin{prp}\label{prp:two-dimensional-orbit}
Let $t\in\R$ and let
$M=\begin{pmatrix}a&b\\c&d\end{pmatrix}\in\rm{GL}_2(\Z)$ be such
that $M\cdot t$ is defined.
Then, for every $k\in\{2,3,4,6\}$, we have
$C_{M\cdot t,k}\sim_M C_{t,k}$.
\end{prp}

\begin{proof}
First suppose that $\det(M)=1$. Define
\[
 \widehat M=
 \begin{pmatrix}
  aI_2&bJ_0\\
  -cJ_0&dI_2
 \end{pmatrix}\in \rm M_4(\Z).
\]
A direct computation shows that $\widehat M\in\rm{SO}(2,2|\Z)$ and
\[
 \widehat M\te_t
 =\bigl(a\te_t+bJ_0\bigr)
   \bigl(-cJ_0\te_t+dI_2\bigr)^{-1}
 =\te_{M\cdot t}.
\]
For every $W\in\rm{SL}_2(\Z)$, the identity $W^tJ_0W=J_0$ gives
$W^tJ_0=J_0W^{-1}$ and $J_0W^t=W^{-1}J_0$.
Using these two relations in the four blocks of $\widehat M$, we
obtain $\widehat M\rh(W^t)=\rh(W^t)\widehat M$. In particular,
$\widehat M\rh(W_k^t)\widehat M^{-1}=\rh(W_k^t)$.
Since $1\notin\operatorname{spec}(W_k)$,
Theorem~\ref{thm:main}, applied with $P=W_k^t$, gives
\[
 C_{M\cdot t,k}\sim_M C_{t,k}.
\]

Now suppose that $\det(M)=-1$. Put
$D=\begin{pmatrix}1&0\\0&-1\end{pmatrix}$ and
$K=\begin{pmatrix}0&1\\1&0\end{pmatrix}$. Then
$M_+=MD\in\rm{SL}_2(\Z)$ and $M_+\cdot(-t)=M\cdot t$. Moreover,
$\rh(K)\te_t=K\te_tK^t=-\te_t$
and, for our choice of the matrices $W_k$, we have
$KW_kK=W_k^{-1}$.
Set $g=\widehat{M_+}\rh(K)$. From the first part of the proof, we
have
\[
 g\te_t=\te_{M\cdot t},\qquad
 g\rh(W_k^t)g^{-1}=\rh\bigl((W_k^{-1})^t\bigr).
\]
Applying Theorem~\ref{thm:main} with $P=(W_k^{-1})^t$, we obtain
\[
 A_{M\cdot t}\rtimes_{\af_{W_k^{-1}}}\Z_k
 \sim_M C_{t,k}.
\]
Put $A=A_{M\cdot t}$ and $\alpha=\af_{W_k}$. Precomposing
the cyclic action with the automorphism $[j]\mapsto[-j]$ of
$\Z_k$ replaces $\alpha$ by $\alpha^{-1}=\af_{W_k^{-1}}$.

Let $u$ and $v$ be the canonical implementing unitaries in
$A\rtimes_\alpha\Z_k$ and $A\rtimes_{\alpha^{-1}}\Z_k$,
respectively. Since
\[
 u^*au=\alpha^{-1}(a)\quad(a\in A),\qquad (u^*)^k=1,
\]
the universal property gives a $*$-homomorphism
\[
 A\rtimes_{\alpha^{-1}}\Z_k\longrightarrow
 A\rtimes_\alpha\Z_k,\qquad a\longmapsto a,\quad v\longmapsto u^*.
\]
Its inverse fixes $A$ and sends $u$ to $v^*$. Thus
\[
 A_{M\cdot t}\rtimes_{\af_{W_k^{-1}}}\Z_k
 \cong C_{M\cdot t,k},
\]
and the result follows.
\end{proof}

We now obtain the complete Morita equivalence classification in
dimension two. For irrational parameters, the necessary part of the
classification was proved in \cite{BCHL}*{Theorem~5.3}.

\begin{thm}
\label{thm:two-dimensional-classification}
Let $s,t\in\R$ and let $k,l\in\{2,3,4,6\}$. The following conditions
are equivalent:
\begin{enumerate}
\renewcommand{\theenumi}{\roman{enumi}}
\renewcommand{\labelenumi}{\textup{(\theenumi)}}
\item $C_{t,k}\sim_M C_{s,l}$;
\item $k=l$ and $A_t\sim_M A_s$;
\item $k=l$ and there exists
$
 M=\begin{pmatrix}a&b\\c&d\end{pmatrix}\in\rm{GL}_2(\Z)
$
such that $ct+d\neq 0$ and
$
 s=M\cdot t=\frac{at+b}{ct+d}.
$
\end{enumerate}
Equivalently, $k=l$ and either $s,t\in\Q$, or $s,t\in\R\setminus\Q$
belong to the same $\rm{GL}_2(\Z)$-orbit.
\end{thm}

\begin{proof}
\textup{(ii) $\Longleftrightarrow$ (iii).}
For irrational $s,t$, this is \cite{Rieffel1981}*{Theorem~4}.
For $t=p/q\in\Q$ in lowest terms, choose $c,d\in\Z$ with
$pc+qd=1$. Then
\[
 M=\begin{pmatrix}q&-p\\c&d\end{pmatrix}\in\rm{SL}_2(\Z),
 \qquad M\cdot t=0,
\]
and \cite{Li2004}*{Theorem~1.1} gives
$A_t\sim_M A_0=C(\T^2)$. Thus all rational parameters lie in
one orbit and their rotation algebras are Morita equivalent.
If one parameter is rational and the other irrational, they lie in
distinct $\rm{GL}_2(\Z)$-orbits. Their rotation algebras are also
not Morita equivalent, since exactly one is simple and simplicity
is preserved under Morita equivalence.

\smallskip
\noindent\textup{(iii) $\Longrightarrow$ (i).}
Since $k=l$ and $s=M\cdot t$,
Proposition~\ref{prp:two-dimensional-orbit} gives
$C_{s,l}=C_{s,k}\sim_M C_{t,k}$.

\smallskip
\noindent\textup{(i) $\Longrightarrow$ (iii).}
Assume $C_{t,k}\sim_M C_{s,l}$. For every $u\in\R$,
\[
 \rank K_0(C_{u,2})=6,\qquad
 \rank K_0(C_{u,3})=8,\qquad
 \rank K_0(C_{u,4})=9,\qquad
 \rank K_0(C_{u,6})=10
\]
by \cite{ELPW}*{Theorem~0.1}. Morita invariance of $K$-theory
therefore gives $k=l$. If $s,t\in\Q$, the orbit calculation
above proves \textup{(iii)}.

Otherwise, by symmetry assume that $t$ is irrational. Let
$\tau_{u,k}$ denote the canonical tracial state on $C_{u,k}$.
By \cite{ELPW}*{Theorem~0.1}, $C_{t,k}$ has a unique tracial
state. Moreover, \cite{ELPW}*{Theorem~4.9}, together with integral shifts of the parameter, gives, for every $u\in\R$,
\begin{equation}\label{eq:trace-range-two-dimensional}
 (\tau_{u,k})_*\bigl(K_0(C_{u,k})\bigr)
 =\frac{1}{k}(\Z+u\Z).
\end{equation}
The Morita correspondence transports $\tau_{s,k}$ to
$\lambda\tau_{t,k}$ for some $\lambda>0$. Compatibility of
this correspondence with the $K_0$-trace pairing gives
\[
 \Z+s\Z=\lambda(\Z+t\Z).
\]
The right-hand group has $\Z$-rank two, so $s$ is also
irrational. The sets $\{1,s\}$ and $\{\lambda,\lambda t\}$
are therefore $\Z$-bases of the same group. Hence there are
$a,b,c,d\in\Z$ with $ad-bc=\pm1$ such that
\[
 1=\lambda(ct+d),\qquad s=\lambda(at+b).
\]
Consequently $ct+d\neq0$ and
$s=(at+b)/(ct+d)$, proving \textup{(iii)}.
\end{proof}

In particular, for a fixed $k\in\{2,3,4,6\}$,
\[
 A_t\rtimes_{\af_{W_k}}\Z_k\sim_M
 A_s\rtimes_{\af_{W_k}}\Z_k
 \quad\Longleftrightarrow\quad
 A_t\sim_M A_s.
\]
The implication from right to left follows from
Theorem~\ref{thm:main} for both rational and irrational parameters.

\subsection{Morita invariance of the canonical flip}

We now consider the canonical flip action in arbitrary dimension. Recall
that, for $\te\in\mathcal T_n(\R)$, the flip is given by
\[
 \af_\te(U_x^\te)=U_{-x}^\te,\qquad x\in\Z^n,
\]
and we write $B_\te=A_\te\rtimes_{\af_\te}\Z_2$. Since
$(-I_n)^t\te(-I_n)=\te$ for every $\te$, this action is defined for
every noncommutative torus. Moreover, $\rh(-I_n)=-I_{2n}$ is central
in $\rm{SO}(n,n|\Z)$. Thus the conjugacy condition in
Theorem~\ref{thm:main} is automatic for the flip.

\begin{cor}\label{cor:flip-orbit}
Let $n\geq2$, let $\te\in\mathcal T_n(\R)$ and let
$g\in\rm{SO}(n,n|\Z)$ be such that $g\te$ is defined. Then
the hypotheses of Theorem~\ref{thm:main} hold with
$W=P=-I_n$ and $k=2$. Consequently,
\[
 B_{g\te}\sim_M B_\te.
\]
\end{cor}

\begin{proof}
Put $W=P=-I_n$ and $k=2$. Then $W^t\te W=\te$, $W^2=I_n$ and
$1\notin\operatorname{spec}(W)$. Also,
\[
 g\rh(W^t)g^{-1}=g(-I_{2n})g^{-1}
 =-I_{2n}=\rh(P).
\]
The result now follows from Theorem~\ref{thm:main}.
\end{proof}

We next combine this corollary with the results of Elliott and
Li~\cite{EL}*{Proposition~3.3}. Recall from \eqref{eq:invariants}
that
\[
 H_\te=\{x\in\Z^n:\te x\in\Z^n\},\qquad
 r_\te=\rank H_\te,\qquad m_\te=n-r_\te.
\]
For $r\geq0$, let $\iota_r$ denote the automorphism of $C(\T^r)$
defined by
\[
 (\iota_r f)(z)=f(z^{-1}),\qquad z\in\T^r,
\]
where inversion is taken in each coordinate. We use the convention
$C(\T^0)=\mathbb C$ and $\iota_0=\operatorname{id}$.

\begin{prp}\label{prp:flip-reduction}
Let $\te\in\mathcal T_n(\R)$, and put $r=r_\te$ and $m=m_\te$.
There exists a nondegenerate $\eta\in\mathcal T_m(\R)$ such that
\[
 A_\te\sim_M C(\T^r)\otimes A_\eta
\]
and
\begin{equation}\label{eq:flip-reduced-crossed-product}
 B_\te\sim_M
 \bigl(C(\T^r)\otimes A_\eta\bigr)
 \rtimes_{\iota_r\otimes\af_\eta}\Z_2.
\end{equation}
When $m=0$, we take $\eta$ to be the empty matrix,
$A_\eta=\mathbb C$ and $\af_\eta=\operatorname{id}$.
\end{prp}

\begin{proof}
By \cite{EL}*{Proposition~3.3}, there exist an integer $q$ with
$0\leq q\leq n$, an element $g\in\rm{SO}(n,n|\Z)$, and a nondegenerate matrix
$\eta\in\mathcal T_{n-q}(\R)$ such that $g\te$ is defined and
\[
 g\te=\begin{pmatrix}0_q&0\\0&\eta\end{pmatrix}.
\]
The hypotheses of Theorem~\ref{thm:main} hold with $W=P=-I_n$
and $k=2$, giving $A_\te\sim_M A_{g\te}$ and
$B_\te\sim_M B_{g\te}$. Since $A_\eta$ is simple, the centre
of $A_{g\te}=C(\T^q)\otimes A_\eta$ is $C(\T^q)$. On the other
hand, the centre of $A_\te$ is $C(\T^r)$. Morita equivalence
preserves the centres of unital algebras, so $q=r$.

The block form of $g\te$ gives a canonical $*$-isomorphism
\[
 A_{g\te}\cong C(\T^r)\otimes A_\eta,
\]
which is easily checked to be $\Z_2$-equivariant for the actions
$\af_{g\te}$ and $\iota_r\otimes\af_\eta$. It therefore induces
an isomorphism of the crossed products. Together with the two
Morita equivalences above, this proves the result.
\end{proof}

\begin{cor}\label{cor:rational-flip}
If $\te\in\mathcal T_n(\Q)$, then
\[
 B_\te\sim_M C(\T^n)\rtimes_{\iota_n}\Z_2.
\]
Thus, for fixed $n$, all flip crossed products associated to rational
parameters are Morita equivalent.
\end{cor}

\begin{proof}
Choose a positive integer $d$ such that $d\te$ has integral entries.
Then $d\Z^n\subseteq H_\te$, so $r_\te=n$ and $m_\te=0$.
The result follows from Proposition~\ref{prp:flip-reduction}.
\end{proof}

The reduction in Proposition~\ref{prp:flip-reduction} will be used
in the next subsection to recover the Morita equivalence class of
the underlying torus from that of its flip crossed product.

\subsection{Descent from flip crossed products}

We show that Morita equivalence of flip crossed products implies
Morita equivalence of the underlying tori. Using
Proposition~\ref{prp:flip-reduction}, we recover the dimension of the
commutative torus factor and the Morita equivalence class of the
simple factor.

\begin{thm}\label{thm:flip-descent}
Let $n\geq2$ and let $\te,\te'\in\mathcal T_n(\R)$. If $B_\te\sim_M B_{\te'},$
then $A_\te\sim_M A_{\te'}.$
\end{thm}
\begin{proof}
Put $r=r_\te$ and $r'=r_{\te'}$. By
Proposition~\ref{prp:flip-reduction}, there are nondegenerate matrices
$\eta\in\mathcal T_{n-r}(\R)$ and
$\eta'\in\mathcal T_{n-r'}(\R)$ such that
\[
 A_\te\sim_M C(\T^r)\otimes A_\eta,
 \qquad
 A_{\te'}\sim_M C(\T^{r'})\otimes A_{\eta'}.
\]
Moreover, $B_\te\sim_M D_{r,\eta}$ and $B_{\te'}\sim_M D_{r',\eta'}$, where
\begin{equation}\label{eq:definition-reduced-flip}
 D_{r,\eta}
 =\bigl(C(\T^r)\otimes A_\eta\bigr)
   \rtimes_{\iota_r\otimes\af_\eta}\Z_2.
\end{equation}
Hence $D_{r,\eta}\sim_M D_{r',\eta'}$, and it suffices to prove
$r=r'$ and $A_\eta\sim_M A_{\eta'}$. We retain the convention
$A_\eta=\mathbb C$ and $\af_\eta=\operatorname{id}$ when $n-r=0$,
and similarly for $\eta'$.

We first determine the centre of $D_{r,\eta}$. For $r>0$, set
\[
 X_r=\T^r/(z\sim z^{-1}),\qquad
 F_r=\{z\in\T^r:z=z^{-1}\}=\{1,-1\}^r.
\]
Let $q_r:\T^r\to X_r$ be the quotient map and $v$ the unitary
implementing the flip. Every element of $D_{r,\eta}$ has a unique
expression $a_0+a_1v$ with $a_i\in C(\T^r,A_\eta)$.
If it is central, commutation with $f\in C(\T^r)$ gives
\[
 \bigl(f(z^{-1})-f(z)\bigr)a_1(z)=0
 \qquad(z\in\T^r).
\]
Since scalar functions separate points, $a_1$ vanishes on the dense
set $\T^r\setminus F_r$, hence everywhere. As $Z(A_\eta)=\mathbb C1$,
commutation with constant $A_\eta$-valued functions forces
$a_0=f\otimes1$ for some $f\in C(\T^r)$, and commutation with $v$
then gives $f(z)=f(z^{-1})$.
Conversely, these invariant scalar functions are central, so
\begin{equation}\label{eq:flip-reduced-centre}
 Z(D_{r,\eta})=C(\T^r)^{\Z_2}\cong C(X_r).
\end{equation}
For $r=0$, $D_{0,\eta}=B_\eta$ is simple by
\cite{ELPW}*{Theorem~6.6 and its proof}, so the same formula holds
with $X_0$ a one-point space.

We next compute $\dim X_r$. For $r>0$, every point of
$\T^r\setminus F_r$ has a closed coordinate cube neighbourhood
disjoint from its image under inversion. By second countability,
choose such cubes $K_1,K_2,\ldots$ whose interiors cover
$\T^r\setminus F_r$. Each restriction $q_r|_{K_j}$ is a homeomorphism
onto a closed copy of $[0,1]^r$ in the compact metrizable space $X_r$.
Thus
\[
 X_r=q_r(F_r)\cup\bigcup_{j=1}^{\infty}q_r(K_j)
\]
is a countable closed cover. Each $q_r(K_j)$ has covering dimension
$r$ by \cite{Engelking1978}*{Corollary~1.8.3}, and $q_r(F_r)$ is finite.
The countable sum theorem gives $\dim X_r\leq r$, while the closed
subspace theorem applied to $q_r(K_1)$ gives the reverse inequality
\cite{Engelking1978}*{Theorems~3.1.8 and 3.1.4}.
Hence $\dim X_r=r$, also for $r=0$.
Morita invariance of the centres and \eqref{eq:flip-reduced-centre}
give $X_r\cong X_{r'}$, so $r=r'$.

If $r=0$, both $\te$ and $\te'$ are nondegenerate, and the conclusion
follows from \cite{ChTrace}*{Corollary~5.11}. Assume henceforth that $r>0$.

To identify the quotients over free orbits, let
\[
 J_x=\{f\in C(X_r):f(x)=0\},\qquad
 I_x=\overline{J_xD_{r,\eta}}
 \qquad(x\in X_r),
\]
where $C(X_r)$ is identified with the centre. For $x=q_r(z)$ with
$z\ne z^{-1}$, define $\pi_z:D_{r,\eta}\to M_2(A_\eta)$ by
\[
 \pi_z(a_0+a_1v)=
 \begin{pmatrix}
  a_0(z)&a_1(z)\\
  \af_\eta(a_1(z^{-1}))&\af_\eta(a_0(z^{-1}))
 \end{pmatrix}.
\]
The assignments
\[
 a\longmapsto\operatorname{diag}(a(z),\af_\eta(a(z^{-1}))),
 \qquad
 v\longmapsto\begin{pmatrix}0&1\\1&0\end{pmatrix}
\]
satisfy covariance, so $\pi_z$ is a $*$-homomorphism.
This map is also surjective. For if $b=(b_{ij})\in M_2(A_\eta)$,
choose $f\in C(\T^r,[0,1])$ with $f(z)=1$ and $f(z^{-1})=0$.
Then the coefficient functions
\[
 a_0=f b_{11}+(1-f)\af_\eta(b_{22}),\qquad
 a_1=f b_{12}+(1-f)\af_\eta(b_{21})
\]
satisfy $\pi_z(a_0+a_1v)=b$, since $\af_\eta^2=\operatorname{id}$.

We now show that $\ker\pi_z=I_x$. Since $\pi_z$ vanishes on $J_x$,
we have $I_x\subseteq\ker\pi_z$. Conversely, let
$a_0+a_1v\in\ker\pi_z$. Then both coefficients vanish at $z$ and
$z^{-1}$. Choose a compatible metric $d$ on $X_r$ and set
\[
 f_j(y)=\min\{1,jd(y,x)\},\qquad f_j\in J_x.
\]
Since $a_i$ vanishes on $q_r^{-1}(x)$, continuity and compactness give
\[
 (f_j\circ q_r)a_i\longrightarrow a_i
 \quad\text{uniformly on }\T^r,\qquad i=0,1.
\]
Thus $f_j(a_0+a_1v)\to a_0+a_1v$ in norm, with each term in $I_x$.
Hence $\ker\pi_z=I_x$ and
\begin{equation}\label{eq:flip-free-fibre}
 D_{r,\eta}/I_x\cong M_2(A_\eta)
 \qquad\bigl(x\in X_r\setminus q_r(F_r)\bigr).
\end{equation}

To compare these quotients, let $\mathcal E$ be a
$D_{r,\eta}$--$D_{r,\eta'}$ imprimitivity bimodule. Its isomorphism
of centres $\Phi$ and the induced homeomorphism $h:X_r\to X_r$ satisfy
\[
 f\xi=\xi\Phi(f),\qquad
 \Phi(f)(h(x))=f(x)
 \qquad(f\in C(X_r),\ \xi\in\mathcal E).
\]
Writing $J'_y$ and $I'_y$ for the corresponding ideals on the other
side, we have $\Phi(J_x)=J'_{h(x)}$ and
\[
 \overline{I_x\mathcal E}
 =\overline{\mathcal E I'_{h(x)}}.
\]
By the Rieffel correspondence
\cite{RW1998}*{Theorem~3.22 and Proposition~3.25},
$\mathcal E/\overline{I_x\mathcal E}$ implements
\[
 D_{r,\eta}/I_x\sim_M D_{r,\eta'}/I'_{h(x)}.
\]

Since $X_r$ is infinite and $q_r(F_r)$ is finite, choose
\[
 x\notin q_r(F_r)\cup h^{-1}(q_r(F_r)).
\]
Both $x$ and $h(x)$ represent free orbits, so \eqref{eq:flip-free-fibre}
gives
\[
 M_2(A_\eta)\sim_M M_2(A_{\eta'}),
 \qquad\text{and hence}\qquad
 A_\eta\sim_M A_{\eta'}.
\]
Tensoring the latter equivalence with $C(\T^r)$ gives
\[
 A_\te
 \sim_M C(\T^r)\otimes A_\eta
 \sim_M C(\T^r)\otimes A_{\eta'}
 \sim_M A_{\te'}.
\]
\end{proof}

Theorem~\ref{thm:flip-descent} gives the following necessary conditions
in terms of the invariants in \eqref{eq:invariants}.

\begin{cor}\label{cor:flip-necessary-invariants}
Let $n\geq2$ and let $\te,\te'\in\mathcal T_n(\R)$.
If $B_\te\sim_M B_{\te'}$, then
\[
 r_\te=r_{\te'},\qquad
 \mathcal R_{\te'}=\lambda\mathcal R_\te
 \quad\text{for some }\lambda>0.
\]
\end{cor}

\begin{proof}
By Theorem~\ref{thm:flip-descent}, $A_\te\sim_M A_{\te'}$.
The centres are therefore isomorphic, giving $r_\te=r_{\te'}$.
The ordered $K_0$-groups are also isomorphic, so
\cite{EL}*{Proposition~4.1} gives the asserted equality of trace
ranges up to a positive scalar.
\end{proof}

\subsection{Classification of flip crossed products}

We now prove that the necessary conditions in
Corollary~\ref{cor:flip-necessary-invariants} are sufficient when
$r_\te\leq1$ or $m_\te\in\{0,2\}$. These cases include every
$\theta$ in dimensions at most four. We retain the notation
$D_{r,\eta}$ from \eqref{eq:definition-reduced-flip}.

The case $r_\te=1$ requires some preparation. For a nondegenerate
$\eta\in\mathcal T_m(\R)$, let
\begin{equation}\label{eq:regular-flip-inclusion}
 j_\eta:B_\eta\longrightarrow M_2(A_\eta),\qquad
 j_\eta(a+bv)=
 \begin{pmatrix}a&b\\\af_\eta(b)&\af_\eta(a)\end{pmatrix},
\end{equation}
where $v$ implements the flip. This is the regular representation
of the crossed product. In particular, it is injective.

\begin{lem}\label{lem:regular-flip-K0}
Let $m\geq2$ and let $\eta\in\mathcal T_m(\R)$ be nondegenerate.
Under the standard matrix Morita identification
$K_0(M_2(A_\eta))\cong K_0(A_\eta)$, the regular inclusion induces
a surjective homomorphism
$(j_\eta)_*:K_0(B_\eta)\longrightarrow K_0(A_\eta)$. Moreover,
\[
 K_0(B_\eta)\cong\Z^{3\cdot2^{m-1}},\qquad
 K_0(A_\eta)\cong\Z^{2^{m-1}},\qquad
 \ker(j_\eta)_*\cong\Z^{2^m}.
\]
Let $\tau_\eta$ and $\widehat\tau_\eta$ be the canonical tracial
states on $A_\eta$ and $B_\eta$, respectively, and put
$\tau_\eta^{(2)}=\frac12\operatorname{Tr}_2\otimes\tau_\eta$,
where $\operatorname{Tr}_2$ is the unnormalized matrix trace.
Then
\[
 \tau_\eta^{(2)}\circ j_\eta=\widehat\tau_\eta.
\]
\end{lem}

\begin{proof}
We first identify $(j_\eta)_*$ with restriction of scalars.
As a right $A_\eta$-module, $B_\eta$ is free with basis
$\{1,v\}$, and left multiplication in this basis gives
\eqref{eq:regular-flip-inclusion}. Thus, for a projection
$p\in M_k(B_\eta)$, the projection $j_\eta^{(k)}(p)$ has range
$pB_\eta^k$, viewed as a right $A_\eta$-module. Consequently,
$(j_\eta)_*$ sends the class of a finitely generated projective
right $B_\eta$-module to the class of its restriction to $A_\eta$.
In particular, $(j_\eta)_*([1_{B_\eta}])=2[1_{A_\eta}]$,
since $B_\eta$ is free of rank two as a right $A_\eta$-module.

To prove surjectivity, we may first replace $\eta$ by a suitable
integral skew-symmetric translate. Indeed, for an integral
skew-symmetric matrix $N$, the isomorphism
$A_\eta\longrightarrow A_{\eta+N}$ in the Weyl presentation is
\[
 U_x^\eta\longmapsto\epsilon_N(x)U_x^{\eta+N},\qquad
 \epsilon_N(x)=(-1)^{\sum_{i<j}N_{ij}x_ix_j}.
\]
Since $\epsilon_N(-x)=\epsilon_N(x)$, this isomorphism intertwines
the flips. It therefore extends to the crossed products and
intertwines their regular representations. Surjectivity is
therefore preserved under these identifications.

By \cite{ChSym}*{Theorem~5.7}, we may now assume that
$K_0(A_\eta)$ is generated by $[1_{A_\eta}]$ and the classes
of the fundamental projective modules. Each such module $E$
admits the involution
\[
 (Tf)(x,t)=f(-x,-t),\qquad (x,t)\in\R^p\times\Z^{m-2p},
\]
on its dense Schwartz space. The covariance identities in
\cite{ChSym}*{equations~(4.2)--(4.4)} show that $T$ extends to
$E$ and is compatible with the flip on $A_\eta$. By
\cite{ChSym}*{Proposition~4.1}, this equips $E$ with a finitely
generated projective right $B_\eta$-module structure whose
restriction to $A_\eta$ is the original module. Hence $[E]$
lies in the image of $(j_\eta)_*$. The unit class also lies
in the image, since
\[
 j_\eta\left(\frac{1+v}{2}\right)
 =\frac12\begin{pmatrix}1&1\\1&1\end{pmatrix},
 \qquad
 (j_\eta)_*\left[\frac{1+v}{2}\right]=[1_{A_\eta}].
\]
Indeed, the constant unitary
$2^{-1/2}\begin{pmatrix}
 1 & 1 \\
 1 & -1
 \end{pmatrix}$
conjugates this projection to $\operatorname{diag}(1,0)$.
Thus $(j_\eta)_*$ is surjective.

The standard $K$-theory computation for noncommutative tori gives
$K_0(A_\eta)\cong\Z^{2^{m-1}}$, while
\cite{ChSym}*{Corollary~4.10} gives
$K_0(B_\eta)\cong\Z^{3\cdot2^{m-1}}$. By surjectivity, we have
the short exact sequence of abelian groups
\[
 0\longrightarrow\ker(j_\eta)_*
 \longrightarrow K_0(B_\eta)
 \xrightarrow{(j_\eta)_*}K_0(A_\eta)
 \longrightarrow0.
\]
Since $K_0(A_\eta)$ is free abelian, this sequence splits.
Its kernel is therefore free of
rank $3\cdot2^{m-1}-2^{m-1}=2^m$.

Finally, since $\tau_\eta\circ\af_\eta=\tau_\eta$, the matrix
formula \eqref{eq:regular-flip-inclusion} gives
\[
 \tau_\eta^{(2)}(j_\eta(a+bv))
 =\frac12\bigl(\tau_\eta(a)+\tau_\eta(\af_\eta(a))\bigr)
 =\tau_\eta(a)=\widehat\tau_\eta(a+bv).
\]
\end{proof}

\begin{prp}\label{prp:one-dimensional-centre}
Let $\eta,\eta'\in\mathcal T_m(\R)$ be nondegenerate, where
$m\geq2$. If $A_\eta\sim_M A_{\eta'}$, then
\[
 D_{1,\eta}\sim_M D_{1,\eta'}.
\]
\end{prp}

\begin{proof}
Put
\[
 A=A_\eta,\quad A'=A_{\eta'},\quad
 B=B_\eta,\quad B'=B_{\eta'},\quad
 E=M_2(A),\quad E'=M_2(A').
\]
Let $j:B\longrightarrow E$ and $j':B'\longrightarrow E'$ be the
regular inclusions. We also use $j,j'$ for their stabilizations
by the compact operators $\mathcal K$, and make the standard
identifications of $K$-groups with those of their stabilizations.

The regular representation gives the interval model
\begin{equation}\label{eq:flip-interval}
 D_{1,\eta}\cong\mathcal D(j)
 :=\{f\in C([0,1],E):f(0),f(1)\in j(B)\}.
\end{equation}
Indeed, for $z_t=e^{\pi i t}$, it sends $f_0+f_1v$, with
$f_0,f_1\in C(\T,A)$, to
\[
 t\longmapsto
 \begin{pmatrix}
 f_0(z_t)&f_1(z_t)\\
 \af_\eta(f_1(z_t^{-1}))&\af_\eta(f_0(z_t^{-1}))
 \end{pmatrix}.
\]
At $t=0,1$, we have $z_t=z_t^{-1}$, so both endpoint values
belong to $j(B)$. Conversely, the four entries of a function in
$\mathcal D(j)$ determine $f_0,f_1$ on the two semicircles,
and the endpoint conditions make them agree at $1$ and $-1$.
Similarly, $D_{1,\eta'}\cong\mathcal D(j')$.

We shall identify $\mathcal D(j)$ with a full corner of
$\mathcal D(j')\otimes\mathcal K$. This requires compatible
identifications of the interior algebra $E$ and the endpoint
algebra $B$. We first construct the endpoint identification
using Lemma~\ref{lem:regular-flip-K0} and AF classification.

Since $A\sim_M A'$, the Brown--Green--Rieffel theorem
\cite{BGR1977}*{Theorem~1.2} gives a $*$-isomorphism
\[
 \Psi:E\otimes\mathcal K\longrightarrow E'\otimes\mathcal K.
\]
Put $h=\Psi_*$. By Lemma~\ref{lem:regular-flip-K0},
$j_*$ and $j'_*$ are split surjections with free kernels of the
same rank. Choose homomorphic sections $s,s'$ and an isomorphism
$\kappa:\ker j_*\longrightarrow\ker j'_*$. The sections give
\[
 K_0(B)=s(K_0(E))\oplus\ker j_*,\qquad
 K_0(B')=s'(K_0(E'))\oplus\ker j'_*.
\]
Indeed, every $g\in K_0(B)$ has the unique decomposition
$g=s(j_*g)+(g-s(j_*g))$, with the second term in $\ker j_*$.
Thus the formula
\[
 H(s(x)+y)=s'(h(x))+\kappa(y),
 \qquad x\in K_0(E),\quad y\in\ker j_*,
\]
defines a group isomorphism $H:K_0(B)\longrightarrow K_0(B')$
satisfying
\begin{equation}\label{eq:regular-K0-square}
 j'_*H=hj_*.
\end{equation}

We next check the order. Let $\tau_E,\tau_{E'},\tau_B,\tau_{B'}$
denote the unique normalized traces; see
\cite{ELPW}*{Proposition~5.7 and Corollary~5.12}.
Set
\[
 q=1_E\otimes e_{11},\qquad P=\Psi(q),\qquad
 \mathsf T'=\tau_{E'}\otimes\operatorname{Tr},
\]
where $\operatorname{Tr}$ is the unnormalized trace on
$\mathcal K$. The trace $\mathsf T'$ is finite on the corner
with unit $P$. To see this, let
$r_N=1_{E'}\otimes\sum_{l=1}^Ne_{ll}$. Since $(r_N)$ is an
approximate identity, $Pr_NP\to P$ in norm. For some $N$,
$\|P-Pr_NP\|<1/2$, and hence
\[
 \tfrac12\mathsf T'(P)
 \leq\mathsf T'(Pr_NP)
 =\mathsf T'(r_NPr_N)
 \leq\mathsf T'(r_N)=N.
\]
The equality uses the trace identity for $r_NP$.
Consequently $c:=\mathsf T'(P)$ is finite, and it is strictly
positive because $P\ne0$ and $\mathsf T'$ is faithful. The formula
\[
 \rho(a)=\mathsf T'\bigl(\Psi(a\otimes e_{11})\bigr),
 \qquad a\in E,
\]
therefore defines a bounded positive trace on $E$, with
$\rho(1_E)=c$; for $a\geq0$, one has
$0\leq\rho(a)\leq c\|a\|$. Uniqueness of the normalized
trace gives $\rho=c\tau_E$. Passing to $K_0$, using the
unnormalized matrix extensions of the traces, gives
\[
 (\tau_{E'})_*h=c(\tau_E)_*.
\]
Together with Lemma~\ref{lem:regular-flip-K0} and
\eqref{eq:regular-K0-square}, this yields
\[
 (\tau_{B'})_*H
 =(\tau_{E'})_*j'_*H
 =(\tau_{E'})_*hj_*
 =c(\tau_B)_*.
\]
The algebras $B,B'$ are simple AF algebras
\cite{ELPW}*{Theorem~6.6}. For such algebras with unique trace,
a nonzero $K_0$-class is positive exactly when its trace is
strictly positive; see also \cite{ELPW}*{proof of Theorem~6.4}.
Since $c>0$, both $H$ and $H^{-1}$ preserve the positive cones.
Thus $H$ is an order isomorphism.

To realize $H$ at the algebra level, we choose a corner whose
unit has class $H([1_B])$. This class is nonzero and positive,
so the definition of the positive cone gives a projection
$p'\in B'\otimes\mathcal K$ with $[p']=H([1_B])$. It is full because
$B'\otimes\mathcal K$ is simple. Put
\[
 C=p'(B'\otimes\mathcal K)p',\qquad
 \iota_B:C\hookrightarrow B'\otimes\mathcal K.
\]
The full-corner inclusion induces an order isomorphism on
$K_0$ \cite{BruceTakeishi}*{Proposition~2.3}. Hence
$(\iota_B)_*^{-1}H:K_0(B)\to K_0(C)$ is an order isomorphism
sending $[1_B]$ to $[1_C]$. Since $C$ is a unital AF algebra,
AF classification \cite{ElliottAF} gives a unital $*$-isomorphism
\[
 \beta:B\longrightarrow C,
 \qquad (\iota_B)_*\beta_*=H.
\]

We now match the interior algebra to this endpoint corner.
Put $q'=j'(p')$ and $F=q'(E'\otimes\mathcal K)q'$.
Since $j$ is unital, \eqref{eq:regular-K0-square} gives
\[
 [P]=h([q])=h(j_*[1_B])=j'_*H([1_B])=[q'].
\]
The algebra $A'$ is AT of real rank zero
\cite{PhillipsAT}*{Theorem~3.8}, hence has stable rank one.
Cancellation of projections
\cite{Blackadar1998}*{Proposition~6.5.1}
therefore gives a partial isometry
$w\in E'\otimes\mathcal K$ satisfying
\[
 w^*w=P,\qquad ww^*=q'.
\]
The map $a\mapsto a\otimes e_{11}$ identifies $E$ with
$q(E\otimes\mathcal K)q$. The isomorphism $\Psi$ carries
this corner onto $P(E'\otimes\mathcal K)P$, and conjugation
by $w$ carries the latter onto $F$, with inverse $x\mapsto w^*xw$.
Thus
\[
 \psi:E\longrightarrow F,\qquad
 \psi(a)=w\Psi(a\otimes e_{11})w^*,
\]
is a unital $*$-isomorphism, with $\psi(1_E)=q'$.
Let $\iota_F:F\hookrightarrow E'\otimes\mathcal K$ be the
inclusion. Conjugation by $w$ preserves the Murray--von Neumann
classes of projections, so $(\iota_F)_*\psi_*=h$ under the
standard stabilization identification. Moreover, $q'\ne0$,
so it is full in the simple algebra $E'\otimes\mathcal K$;
in particular, $(\iota_F)_*$ is an order isomorphism.

The maps $\beta$ and $\psi$ identify the endpoint and interior
algebras, respectively. To obtain an isomorphism of the interval
algebras, we must also compare the two endpoint embeddings in $F$.
Define the unital monomorphisms
\[
 \varphi_1=\psi\circ j,\qquad
 \varphi_2=j'\circ\beta:B\longrightarrow F.
\]
The preceding identities imply
\[
 (\iota_F)_*(\varphi_1)_*=hj_*=j'_*H
                         =(\iota_F)_*(\varphi_2)_*.
\]
Injectivity of $(\iota_F)_*$ gives
$(\varphi_1)_*=(\varphi_2)_*:K_0(B)\to K_0(F)$.
Since $B$ is AF, it satisfies the universal coefficient theorem
\cite{RosenbergSchochet1987}. The freeness of $K_0(B)$ and
$K_1(B)=0$ make its extension terms vanish, giving the natural
identification
\[
 KK(B,F)\cong\operatorname{Hom}(K_0(B),K_0(F)).
\]
Consequently $[\varphi_1]=[\varphi_2]$ in $KK(B,F)$.
For every tracial state $\tau$ on $F$, the compositions
$\tau\circ\varphi_1$ and $\tau\circ\varphi_2$ are tracial
states on $B$. Uniqueness of the trace on $B$ therefore gives
$\tau\circ\varphi_1=\tau\circ\varphi_2$.
Finally, the rotation-related obstruction in
\cite{Lin}*{Definition~3.4} is zero: the degree-one part of
the required splitting has domain $K_1(B)=0$, so its
composition with the rotation homomorphism vanishes.

The source $B$ is unital AF, hence unital AH, and the target
$F\cong M_2(A)$ is unital, separable, simple and of tracial
rank zero by \cite{PhillipsAT}*{Theorem~3.5} and matrix
stability of tracial rank zero.
All the hypotheses of \cite{Lin}*{Theorem~10.7} are now
satisfied. We obtain a norm-continuous unitary path
$u:[0,\infty)\to U(F)$ such that
\begin{equation}\label{eq:regular-asymptotic}
 \lim_{s\to\infty}
 \|u(s)^*\varphi_1(b)u(s)-\varphi_2(b)\|=0,
 \qquad b\in B.
\end{equation}

We use this path to obtain the required endpoint values.
Put $\ell(t)=1/(t(1-t))$ for $0<t<1$ and let
\[
 \mathcal G=\{g\in C([0,1],F):g(0),g(1)\in\varphi_2(B)\}.
\]
For $f\in\mathcal D(j)$, write $f(i)=j(b_i)$, $i=0,1$,
where $b_i$ is unique because $j$ is injective, and define
\[
 \Phi(f)(t)=
 \begin{cases}
 u(\ell(t))^*\psi(f(t))u(\ell(t)),&0<t<1,\\
 \varphi_2(b_i),&t=i,\quad i=0,1.
 \end{cases}
\]
This function is continuous in the interior. At either
endpoint $i$, write $u_t=u(\ell(t))$. Since
$\psi(f(i))=\varphi_1(b_i)$, isometry of $\psi$ and unitary
conjugation give
\[
 \begin{aligned}
 \|\Phi(f)(t)-\varphi_2(b_i)\|
 &\leq\|f(t)-f(i)\| + \|u_t^*\varphi_1(b_i)u_t-\varphi_2(b_i)\|
 \longrightarrow0.
 \end{aligned}
\]
The first term tends to zero by continuity of $f$ and the
second by \eqref{eq:regular-asymptotic}, since $\ell(t)\to\infty$
at both endpoints. Thus $\Phi(f)\in\mathcal G$.
For each $0<t<1$, the formula is a $*$-isomorphism $E\to F$,
and at the endpoints it is the $*$-isomorphism
$\varphi_2\circ j^{-1}:j(B)\to\varphi_2(B)$.
Since the algebra operations on continuous functions are
pointwise, $\Phi:\mathcal D(j)\to\mathcal G$ is a $*$-homomorphism.

For completeness, let $g\in\mathcal G$ and write
$g(i)=\varphi_2(c_i)$. Define
\[
 f(t)=
 \begin{cases}
 \psi^{-1}\bigl(u_tg(t)u_t^*\bigr),&0<t<1,\\
 j(c_i),&t=i,\quad i=0,1.
 \end{cases}
\]
At either endpoint,
\[
 \begin{aligned}
 \|f(t)-j(c_i)\|
 &=\|g(t)-u_t^*\varphi_1(c_i)u_t\|\\
 &\leq\|g(t)-g(i)\|
       +\|\varphi_2(c_i)-u_t^*\varphi_1(c_i)u_t\|
 \longrightarrow0.
 \end{aligned}
\]
Hence $f\in\mathcal D(j)$. Substitution into the two formulas
gives the identity at every $t\in[0,1]$, proving that
$\Phi:\mathcal D(j)\to\mathcal G$ is an isomorphism.

It remains to identify $\mathcal G$ as a full corner. We have
\[
 \begin{aligned}
 R&:=\mathcal D(j')\otimes\mathcal K=\{g\in C([0,1],E'\otimes\mathcal K):
             g(0),g(1)\in j'(B'\otimes\mathcal K)\}.
 \end{aligned}
\]
This identification follows by tensoring the extension given
by endpoint evaluation with the exact algebra $\mathcal K$.
Regard $q'$ as a constant projection in $R$; it belongs to $R$
because $q'=j'(p')$. Compression by $q'$ gives interior values
in $F$ and endpoint values in
\[
 q'j'(B'\otimes\mathcal K)q'
 =j'\bigl(p'(B'\otimes\mathcal K)p'\bigr)
 =\varphi_2(B).
\]
Conversely, every $g\in\mathcal G$ belongs to $R$ and satisfies
$q'gq'=g$. Thus $\mathcal G=q'Rq'$.

To prove fullness, let $J$ be the closed two-sided ideal of
$R$ generated by $q'$. Constant functions define a unital
$*$-homomorphism
\[
 \lambda:B'\longrightarrow\mathcal D(j'),\qquad
 \lambda(b)(t)=j'(b).
\]
Its stabilization $\lambda_s=\lambda\otimes\operatorname{id}_{\mathcal K}$
sends $p'$ to the constant projection $q'$. Since $p'$ is full
in $B'\otimes\mathcal K$, every element of that algebra is a
norm limit of finite sums $\sum_l a_lp'b_l$.
Applying $\lambda_s$ shows that
$\lambda_s(B'\otimes\mathcal K)\subseteq J$. In particular,
\[
 e:=\lambda_s(1_{B'}\otimes e_{11})
   =1_{\mathcal D(j')}\otimes e_{11}\in J.
\]
For $f\in\mathcal D(j')$ and matrix units $e_{ij}\in\mathcal K$,
the ideal property gives
\[
 f\otimes e_{ij}
 =(f\otimes e_{i1})\,e\,
   (1_{\mathcal D(j')}\otimes e_{1j})\in J.
\]
Finite sums of these elements are dense in $R$, so $J=R$.
Therefore $q'$ is full. The full-corner Morita equivalence
\cite{BGR1977}*{Section~1} now gives
\[
 \mathcal D(j)\cong\mathcal G=q'Rq'
 \sim_M R\sim_M\mathcal D(j').
\]
Together with \eqref{eq:flip-interval}, this proves the result.
\end{proof}

The cases in the following classification correspond to the
arguments now available for the reduced algebras $D_{r,\eta}$.
For $r=0$, the simple case is known, and for $r=1$ we use
Proposition~\ref{prp:one-dimensional-centre}. For arbitrary $r$,
the cases $m=0$ and $m=2$ are handled by the commutative case
and the two-dimensional orbit construction, respectively.
These are precisely the alternatives in the definition of
$\mathcal C_n$ below. Recall that
$\mathcal R_\te=(\tau_\te)_*(K_0(A_\te))$.

\begin{thm}\label{thm:flip-classification-Cn}
Fix $n\geq2$ and put
\[
 \mathcal C_n=
 \{\te\in\mathcal T_n(\R):r_\te\leq1
                   \text{ or }m_\te\in\{0,2\}\}.
\]
For $\te,\te'\in\mathcal C_n$, the following conditions are
equivalent:
\begin{enumerate}
\renewcommand{\theenumi}{\roman{enumi}}
\renewcommand{\labelenumi}{\textup{(\theenumi)}}
\item $B_\te\sim_M B_{\te'}$;
\item $A_\te\sim_M A_{\te'}$;
\item $r_\te=r_{\te'}$ and
$\mathcal R_{\te'}=\lambda\mathcal R_\te$ for some $\lambda>0$.
\end{enumerate}
\end{thm}

\begin{proof}
\textup{(i)$\Rightarrow$(ii).}
This is Theorem~\ref{thm:flip-descent}.

\textup{(ii)$\Leftrightarrow$(iii).}
By \cite{EL}*{Theorem~1.1}, the tori $A_\te$ and $A_{\te'}$
are Morita equivalent exactly when their centres and ordered
$K_0$-groups are isomorphic. Their centres are $C(\T^{r_\te})$
and $C(\T^{r_{\te'}})$, so the first requirement is
$r_\te=r_{\te'}$. Since both tori have dimension $n$,
\cite{EL}*{Proposition~4.1} identifies the second requirement
with
\[
 \mathcal R_{\te'}=\lambda\mathcal R_\te
 \quad\text{for some }\lambda>0.
\]
Together, these are precisely the conditions in \textup{(iii)}.

\textup{(ii)$\Rightarrow$(i).}
Suppose that $A_\te\sim_M A_{\te'}$, and put
$r=r_\te=r_{\te'}$ and $m=n-r$.
By Proposition~\ref{prp:flip-reduction}, choose
nondegenerate $\eta,\eta'\in\mathcal T_m(\R)$ such that
\[
 A_\te\sim_M C(\T^r)\otimes A_\eta,
 \qquad
 A_{\te'}\sim_M C(\T^r)\otimes A_{\eta'},
\]
and $B_\te\sim_M D_{r,\eta}$,
$B_{\te'}\sim_M D_{r,\eta'}$. The two tensor products above are
Morita equivalent. Taking corresponding quotients at points of
their centres, as in the proof of Theorem~\ref{thm:flip-descent},
gives $A_\eta\sim_M A_{\eta'}$, since all their fibres are
$A_\eta$ and $A_{\eta'}$, respectively.
It therefore remains to prove
$D_{r,\eta}\sim_M D_{r,\eta'}$.

First suppose that $r\leq1$. If $r=0$, the desired equivalence
follows from \cite{ChTrace}*{Corollary~5.11}. If $r=1$, then
$m\geq2$ and
Proposition~\ref{prp:one-dimensional-centre} gives
$D_{1,\eta}\sim_M D_{1,\eta'}$.

Now suppose that $r\geq2$. Since $\te\in\mathcal C_n$, we have
$m\in\{0,2\}$. If $m=0$, both reduced systems are
$(C(\T^n),\iota_n)$, so their crossed products are isomorphic.

It remains to consider $m=2$. Write $\eta=tJ_0$ and
$\eta'=sJ_0$, where $s,t$ are irrational. By the equivalence of
\textup{(ii)} and \textup{(iii)} in
Theorem~\ref{thm:two-dimensional-classification}, there exists
$M\in\rm{GL}_2(\Z)$ such that $s=M\cdot t$.
The construction in the proof of
Proposition~\ref{prp:two-dimensional-orbit} gives an
$h\in\rm{SO}(2,2|\Z)$ such that $h\eta=\eta'$. Write
$h=\begin{pmatrix}A&B\\C&D\end{pmatrix}$, with
$A,B,C,D\in\rm M_2(\Z)$, and set
\[
 g=
 \begin{pmatrix}
  I_r&0&0&0\\
  0&A&0&B\\
  0&0&I_r&0\\
  0&C&0&D
 \end{pmatrix}
 \in\rm{SO}(n,n|\Z).
\]
Then
\[
 g(0_r\oplus\eta)
 =0_r\oplus(A\eta+B)(C\eta+D)^{-1}
 =0_r\oplus\eta'.
\]
Corollary~\ref{cor:flip-orbit} now gives
$D_{r,\eta}\sim_M D_{r,\eta'}$. This proves
\textup{(ii)}$\Rightarrow$\textup{(i)} and completes the proof.
\end{proof}

\begin{cor}
\label{cor:flip-classification-four}
Let $2\leq n\leq4$ and let $\te,\te'\in\mathcal T_n(\R)$. Then
\[
 B_\te\sim_M B_{\te'}
 \quad\Longleftrightarrow\quad
 A_\te\sim_M A_{\te'}
\]
and these conditions hold exactly when
\[
 r_\te=r_{\te'},\qquad
 \mathcal R_{\te'}=\lambda\mathcal R_\te
 \quad\text{for some }\lambda>0.
\]
\end{cor}

\begin{proof}
The dimension $m_\te$ of the simple factor cannot be one: the only
$1\times1$ skew-symmetric matrix is zero, and its torus is $C(\T)$,
which is not simple. If $r_\te\geq2$ and $n\leq4$, then
$m_\te=n-r_\te\leq2$, and hence $m_\te\in\{0,2\}$.
Otherwise $r_\te\leq1$. Therefore every
$\te\in\mathcal T_n(\R)$ belongs to $\mathcal C_n$, and the
result follows from Theorem~\ref{thm:flip-classification-Cn}.
\end{proof}

For these dimensions the trace ranges in the classification are
explicit. If $\te=(\te_{ij})$, then \eqref{eq:invariants} gives
\[
 \begin{aligned}
 n=2:\quad &\mathcal R_\te=\Z+\te_{12}\Z,\\
 n=3:\quad &\mathcal R_\te=\Z+\te_{12}\Z+\te_{13}\Z+\te_{23}\Z,\\
 n=4:\quad &\mathcal R_\te=\Z+\sum_{1\leq i<j\leq4}\te_{ij}\Z
  +(\te_{12}\te_{34}-\te_{13}\te_{24}+\te_{14}\te_{23})\Z.
 \end{aligned}
\]
Together with $r_\te=\rank\{x\in\Z^n:\te x\in\Z^n\}$, these
formulas give the stated classification entirely in terms of the
parameters. Theorem~\ref{thm:flip-descent} and
Theorem~\ref{thm:flip-classification-Cn} also complete the proof of
Theorem~\ref{thm:classification}.

\subsection{Infinite families for a fixed cyclic group action}

We now construct infinite families of parameters for which the
same matrix $W$ gives Morita equivalent crossed products. Throughout
this subsection, we assume that
\[
 W\in\rm{GL}_n(\Z),\qquad W^k=I_n,\qquad
 1\notin\operatorname{spec}(W).
\]
The point is to choose the matrices in Theorem~\ref{thm:main} from
the centralizer of $\rh(W^t)$. Then the action on the transformed
torus is again given by $W$.

\begin{prp}\label{prp:centralizer-families}
Let $W\in\rm{GL}_n(\Z)$ satisfy $W^k=I_n$ and
$1\notin\operatorname{spec}(W)$. Let
$\te\in\mathcal T_n(\R)$ satisfy $W^t\te W=\te$, and let
$N,C\in\mathcal T_n(\R)\cap M_n(\Z)$ satisfy
\[
 W^tNW=N,\qquad WCW^t=C.
\]
For $a,b\in\Z$ such that $I_n+bC\te$ is invertible, put
\[
 \te_{a,b}=\te(I_n+bC\te)^{-1}+aN.
\]
Then $W^t\te_{a,b}W=\te_{a,b}$ and
\[
 A_{\te_{a,b}}\rtimes_{\af_W}\Z_k
 \sim_M A_\te\rtimes_{\af_W}\Z_k.
\]
\end{prp}

\begin{proof}
Put
\[
 u_a=\begin{pmatrix}I_n&aN\\0&I_n\end{pmatrix},\qquad
 v_b=\begin{pmatrix}I_n&0\\bC&I_n\end{pmatrix}.
\]
Since $N$ and $C$ are integral and skew-symmetric, both matrices
belong to $\rm{SO}(n,n|\Z)$. The invariance conditions give
\[
 NW^{-1}=W^tN,\qquad CW^t=W^{-1}C.
\]
Thus both $u_a$ and $v_b$ commute with
$\rh(W^t)=\operatorname{diag}(W^t,W^{-1})$. Consequently,
\[
 g_{a,b}:=u_av_b
 =\begin{pmatrix}I_n+abNC&aN\\bC&I_n\end{pmatrix}
\]
satisfies
\[
 g_{a,b}\rh(W^t)g_{a,b}^{-1}=\rh(W^t).
\]
Moreover,
\[
 \begin{split}
 g_{a,b}\te
 &=\bigl((I_n+abNC)\te+aN\bigr)(I_n+bC\te)^{-1}\\
 &=\te(I_n+bC\te)^{-1}+aN=\te_{a,b}.
 \end{split}
\]
Theorem~\ref{thm:main}, with $P=W^t$, proves both assertions.
\end{proof}

We next show that this construction always produces infinitely
many distinct parameters under the hypotheses above. This includes
the case $\te=0$.

\begin{thm}\label{thm:infinite-fixed-action}
Let $n\geq2$, let $W\in\rm{GL}_n(\Z)$ satisfy $W^k=I_n$ and
$1\notin\operatorname{spec}(W)$, and let
$\te\in\mathcal T_n(\R)$ satisfy $W^t\te W=\te$.
There are pairwise distinct matrices
$\te_j\in\mathcal T_n(\R)$, indexed by an infinite subset of
$\Z$, such that
\[
 W^t\te_jW=\te_j,\qquad
 A_{\te_j}\rtimes_{\af_W}\Z_k
 \sim_M A_\te\rtimes_{\af_W}\Z_k.
\]
If $\te\neq0$, one can take
\begin{equation}\label{eq:infinite-parameters}
 \te_j=\te(I_n+jC\te)^{-1},
\end{equation}
where $C$ is an explicitly chosen integral skew-symmetric matrix
satisfying $WCW^t=C$. This formula is defined for all but finitely
many integers $j$.
\end{thm}

\begin{proof}
Suppose first that $\te\neq0$. Choose $p<q$ with $\te_{pq}\neq0$
and, writing $E_{ij}$ for the matrix units, set
\[
 C_0=E_{qp}-E_{pq},\qquad
 C=\sum_{l=0}^{k-1}W^lC_0(W^l)^t.
\]
Then $C$ is integral and skew-symmetric, and $W^k=I_n$ gives
$WCW^t=C$. By invariance of $\te$ and cyclicity of the matrix trace,
\[
 \operatorname{Tr}(C\te)
 =\sum_{l=0}^{k-1}\operatorname{Tr}\bigl(C_0(W^l)^t\te W^l\bigr)
 =2k\te_{pq}\neq0.
\]
Thus $\te C\te\neq0$: otherwise $(C\te)^2=0$, forcing
$\operatorname{Tr}(C\te)=0$.

Since $\det(I_n+XC\te)$ is a polynomial with constant term one,
\[
 S=\{j\in\Z:\det(I_n+jC\te)\neq0\}
\]
contains all but finitely many integers. For $j\in S$, define
$\te_j$ by~\eqref{eq:infinite-parameters}.
Proposition~\ref{prp:centralizer-families}, with $a=0$ and $b=j$,
gives $W^t\te_jW=\te_j$ and the required Morita equivalence.
For $j,l\in S$, the resolvent identity and the fact that
$(I_n+jC\te)^{-1}$ commutes with $C\te$ give
\[
 \te_j-\te_l
 =(l-j)\te C\te
       (I_n+jC\te)^{-1}(I_n+lC\te)^{-1}.
\]
Since $\te C\te\neq0$ and both inverse factors are invertible,
this difference is nonzero whenever $j\neq l$.

Now suppose that $\te=0$. We construct a nonzero integral
skew-symmetric matrix $N$ satisfying $W^tNW=N$.
If $W^2\neq I_n$, put
\[
 G=\sum_{l=0}^{k-1}(W^l)^tW^l,\qquad N=G(W-W^{-1}).
\]
The matrix $G$ is integral and positive definite, with $W^tGW=G$.
Hence $W^tG=GW^{-1}$ and $W^{-t}G=GW$, so
\[
 N^t=(W^t-W^{-t})G=-N,\qquad W^tNW=N.
\]
Moreover, $N\neq0$ because $G$ is invertible and $W\neq W^{-1}$.
If $W^2=I_n$, the absence of the eigenvalue $1$ implies $W=-I_n$;
then take $N=E_{12}-E_{21}$.

Set $\te_j=2jN$ for $j\in\Z$. These parameters are pairwise
distinct and $W$-invariant, and
\[
 \sm_{\te_j}(x,y)=\exp(2\pi i jx^tNy)=1=\sm_0(x,y),
 \qquad x,y\in\Z^n.
\]
Thus $U_x^0\mapsto U_x^{\te_j}$ defines a $W$-equivariant
$*$-isomorphism $A_0\cong A_{\te_j}$, and consequently
\[
 A_{\te_j}\rtimes_{\af_W}\Z_k
 \cong A_0\rtimes_{\af_W}\Z_k.
\]
\end{proof}

\begin{rmk}\label{rmk:infinite-family-scope}
The distinctness assertion in Theorem~\ref{thm:infinite-fixed-action}
concerns the parameter matrices; it does not imply that the crossed
products are pairwise non-isomorphic. Indeed, the nonzero invariant
integral skew-symmetric matrix $N$ constructed in the proof depends
only on $W$. For every $W$-invariant $\te$ and $j\in\Z$,
$\sm_{\te+2jN}=\sm_\te$, so the map
$U_x^\te\mapsto U_x^{\te+2jN}$ intertwines the $W$-actions and gives
\[
 A_{\te+2jN}\rtimes_{\af_W}\Z_k
 \cong A_\te\rtimes_{\af_W}\Z_k.
\]
These shifted parameters are pairwise distinct. For $\te\neq0$,
the theorem also provides the fractional-linear
family~\eqref{eq:infinite-parameters} in the same Morita equivalence class.

The hypothesis $1\notin\operatorname{spec}(W)$ cannot be omitted
from an assertion covering every invariant parameter. For example,
if $W=\operatorname{diag}(1,-1)$, then
$W^t(tJ_0)W=-tJ_0$, so zero is the only invariant skew-symmetric
parameter.

Examples~\ref{exa:infinite-block-families} and
\ref{exa:jeong-lee-families} apply the fractional-linear construction
to block actions and to the canonical actions of orders $5,8,10,12$
on $A_t\otimes A_t$ constructed by Jeong--Lee
\cite{JL}*{Proposition~6.4}, respectively.
\end{rmk}

\begin{exa}\label{exa:infinite-block-families}
Let $d\geq1$, $n=2d$ and $t_1,\ldots,t_d\in\R$. Put
\[
 \te=\operatorname{diag}(t_1J_0,\ldots,t_dJ_0),\qquad
 W=\operatorname{diag}(W_{k_1},\ldots,W_{k_d}),
\]
where $k_i\in\{2,3,4,6\}$ and $k$ is the least common multiple of the
$k_i$. Then $W^k=I_{2d}$, $W^t\te W=\te$ and
$1\notin\operatorname{spec}(W)$. Take
\[
 C=-\operatorname{diag}(J_0,\ldots,J_0).
\]
Since every $W_{k_i}$ has determinant one, $WCW^t=C$, and
$C\te=\operatorname{diag}(t_1I_2,\ldots,t_dI_2)$.
For each $j\in\Z$ with $1+jt_i\neq0$ for every $i$, put
\[
 \te_j=
 \operatorname{diag}\left(
 \frac{t_1}{1+jt_1}J_0,\ldots,
 \frac{t_d}{1+jt_d}J_0\right).
\]
Proposition~\ref{prp:centralizer-families}, with $N=0$, $a=0$
and $b=j$, gives
\[
 A_{\te_j}\rtimes_{\af_W}\Z_k
 \sim_M A_\te\rtimes_{\af_W}\Z_k.
\]
Only finitely many integers fail the denominator conditions.
If some $t_i\neq0$, the corresponding scalar
$t_i/(1+jt_i)$ is injective as a function of the admissible $j$,
so the matrices $\te_j$ are pairwise distinct.
If every $t_i\geq0$, all $j\geq0$ are allowed. Taking every
$k_i=2$ gives explicit examples for the flip in every even
dimension.
\end{exa}

\begin{exa}
\label{exa:jeong-lee-families}
Let $t\in\R\setminus\Q$, and put
\(
 J_4=\operatorname{diag}(J_0,J_0),\qquad
 \Th_t=tJ_4.
\)
Then $A_{\Th_t}\cong A_t\otimes A_t$. Jeong--Lee
\cite{JL}*{Proposition~6.4} construct canonical actions of
$\Z_k$, for $k\in\{5,8,10,12\}$, on this torus. Their
matrices in \cite{JL}*{equation~(6.25)}, denoted here by $V_k$, are
\[
 \begin{gathered}
 V_5=\begin{pmatrix}
 0&-1&0&-1\\
 1&-1&0&-1\\
 0&0&0&-1\\
 -1&0&1&0
 \end{pmatrix},\qquad
 V_8=\begin{pmatrix}
 0&0&0&-1\\
 0&0&1&0\\
 1&0&0&0\\
 0&1&0&0
 \end{pmatrix},\\[4pt]
 V_{10}=\begin{pmatrix}
 0&-1&0&-1\\
 1&1&0&1\\
 0&0&0&-1\\
 -1&0&1&0
 \end{pmatrix},\qquad
 V_{12}=\begin{pmatrix}
 0&0&0&-1\\
 0&0&1&1\\
 1&0&0&0\\
 0&1&0&0
 \end{pmatrix}.
 \end{gathered}
\]
A direct computation gives $V_k^tJ_4V_k=J_4$ for every such $k$.
Moreover, as in \cite{JL}*{Lemma~6.3 and equation~(6.25)}, $V_k$
is conjugate in $\rm{GL}_4(\Z)$ to the companion matrix of the
$k$th cyclotomic polynomial. Its eigenvalues are therefore the
primitive $k$th roots of unity. In particular,
\[
 V_k^k=I_4,\qquad 1\notin\operatorname{spec}(V_k),\qquad
 V_k^t\Th_tV_k=\Th_t.
\]
Using the Weyl convention of the present paper, the corresponding
action is written as
\[
 \af_{V_k}(U_x^{\Th_t})=U_{V_kx}^{\Th_t},
 \qquad x\in\Z^4.
\]

Put $C=-J_4=J_4^{-1}$. Then $V_kCV_k^t=C$ and
$C\Th_t=tI_4$. Since $t$ is irrational, $1+jt\neq0$ for
every $j\in\Z$, so
\[
 \begin{aligned}
 I_4+jC\Th_t&=(1+jt)I_4,\\
 \Th_t(I_4+jC\Th_t)^{-1}&=t_jJ_4,
 \qquad t_j:=\frac{t}{1+jt}.
 \end{aligned}
\]
Each $t_j$ is irrational, since $1-jt_j=(1+jt)^{-1}\neq0$
and $t=t_j/(1-jt_j)$. Hence $t_jJ_4$ is nondegenerate.
Moreover,
\[
 t_j-t_l=\frac{(l-j)t^2}{(1+jt)(1+lt)}\neq0
 \qquad(j\neq l).
\]
Thus the parameters $t_jJ_4$ are pairwise distinct.
Proposition~\ref{prp:centralizer-families}, with $N=0$, $a=0$
and $b=j$, now gives
\[
 (A_{t_j}\otimes A_{t_j})\rtimes_{\af_{V_k}}\Z_k
 \sim_M
 (A_t\otimes A_t)\rtimes_{\af_{V_k}}\Z_k,
 \qquad k\in\{5,8,10,12\},\quad j\in\Z.
\]
Here the action on each tensor product is understood through the
identification $A_{t_j}\otimes A_{t_j}\cong A_{t_jJ_4}$.
For each fixed $k$, the matrix $V_k$ is unchanged throughout this
family of simple four-dimensional tori. The orders $5,8,10,12$
do not occur for finite cyclic subgroups of $\rm{SL}_2(\Z)$.
\end{exa}

\appendix
\section{Proofs for the symplectic construction and equivariant bimodule}\label{app:symplectic-computations}

We give the verifications used in Sections~\ref{sec:symplectic-construction} and~\ref{sec:equivariant-bimodule}, retaining their notation
and the embedding maps from Section~\ref{Prelim}.

\subsection{}\label{app:block-identities}

Write $g'=\begin{pmatrix}A'&B'\\C'&D'\end{pmatrix}$.
Left multiplication by $\nu(K)$ leaves the lower block row
unchanged. Comparing that row in
$g''\rho(W^t)=\rho(P_1)g''$ gives
\begin{equation}\label{C',D' relations}
C'W^t=(P_1^t)^{-1}C',\qquad
D'W^{-1}=(P_1^t)^{-1}D'.
\end{equation}

\begin{proof}[Proof of Lemma~\ref{W and A compatibility}]
Since $W^t\te W=\te$, we have $W^t\te=\te W^{-1}$.
Combining this with~\eqref{C',D' relations} gives
\[
\begin{aligned}
(P_1^t)^{-1}(C'\te+D')
&=C'W^t\te+D'W^{-1}=(C'\te+D')W^{-1}.
\end{aligned}
\]
Li's identity $\A^{-1}=C'\te+D'$ therefore implies
$(P_1^t)^{-1}\A^{-1}=\A^{-1}W^{-1}$. Taking inverses
yields $W\A=\A P_1^t$. Substituting $C'=\A^{-1}\Ph$
into the first equation of~\eqref{C',D' relations} then gives
\[
\A^{-1}\Ph W^t
=(P_1^t)^{-1}\A^{-1}\Ph
=\A^{-1}W^{-1}\Ph.
\]
Cancelling $\A^{-1}$ and multiplying by $W$ proves
$W\Ph W^t=\Ph$.
\end{proof}

\begin{proof}[Proof of Proposition~\ref{prop:W-block}]
Initially write $W=\begin{pmatrix}W_1&W_2\\W_3&W_4\end{pmatrix}$.
The identity $W\Ph W^t=\Ph$ reads
\[
\begin{pmatrix}
W_1F_0W_1^t&W_1F_0W_3^t\\
W_3F_0W_1^t&W_3F_0W_3^t
\end{pmatrix}
=\begin{pmatrix}F_0&0\\0&0\end{pmatrix}.
\]
The upper-left block and invertibility of $F_0$ imply
$(\det W_1)^2=1$. The lower-left block then forces
$W_3=0$. Since $W$ is unimodular, so is $W_4$.
Inverting $W_1F_0W_1^t=F_0$ and using
$F_0^{-1}=\te_{11}-Z$ gives
\begin{equation}\label{te-Z eq}
W_1^t(\te_{11}-Z)W_1=\te_{11}-Z.
\end{equation}
The $(1,1)$ block of $W^t\te W=\te$ is now
$W_1^t\te_{11}W_1=\te_{11}$. Subtracting~\eqref{te-Z eq}
gives $W_1^tZW_1=Z$.
\end{proof}

The remaining block identities needed for the construction follow
by multiplying out $W^t\te W=\te$:
\begin{equation}\label{eq:theta-block-identities}
\begin{aligned}
W_2^t\te_{11}W_1+W_4^t\te_{21}W_1&=\te_{21},\\
W_2^t\te_{11}W_2+W_2^t\te_{12}W_4
+W_4^t\te_{21}W_2+W_4^t\te_{22}W_4&=\te_{22}.
\end{aligned}
\end{equation}

\subsection{}\label{app:real-finite}

\begin{proof}[Proof of Lemma~\ref{lem:finite-quotient}]
Write $R_Za=(u,v,w)\in\Z^r\oplus\Z^r\oplus\Z^{2p-2r}$.
The definition of $E_r$ gives
\[
\phi_F(a)=\bigl([m_1u_1]_{n_1},\ldots,[m_ru_r]_{n_r},
[v_1]_{n_1},\ldots,[v_r]_{n_r}\bigr).
\]
Because $\gcd(m_j,n_j)=1$, this vanishes exactly when
$n_j\mid u_j$ and $n_j\mid v_j$ for every $j$.
On the other hand, the normal form~\eqref{eq:Z-normal} gives
\[
Za=R_Z^t
\begin{pmatrix}
\diag(m_j/n_j)v\\-\diag(m_j/n_j)u\\0
\end{pmatrix}.
\]
Unimodularity of $R_Z$ shows that $Za$ is integral exactly
under the same divisibility conditions. Thus
$\ker\phi_F=\Ld_Z$.

Multiplication by $m_j$ is a bijection of $\Z/n_j\Z$.
Hence all coordinates of $\phi_F(a)$ may be prescribed
independently by choosing $u,v$; unimodularity of $R_Z$
then gives a corresponding $a\in\Z^{2p}$. Thus $\phi_F$
is surjective, and the quotient description follows.
\end{proof}

For $R_Za=(u,v,w)$ and $R_Za'=(u',v',w')$, the finite
bicharacter satisfies
\[
\begin{aligned}
\rh_F(\phi_F(a),\phi_F(a'))
&=\exp\!\left(2\pi i\sum_{j=1}^r
\frac{m_j}{n_j}(u_jv_j'-v_ju_j')\right)=\exp(2\pi i a^tZa'),
\end{aligned}
\]
proving~\eqref{eq:rhoF-Z}. The real factor contributes
$\exp(2\pi i a^t(\te_{11}-Z)a')$ because
$T_{11}^tJ_0T_{11}=\te_{11}-Z$. Multiplying these
contributions proves~\eqref{eq:rho0-tau}; in particular,
\begin{equation}\label{eq:rho0-A-invariance}
\rh_0(\ta(W_1a),\ta(W_1a'))
=\rh_0(\ta(a),\ta(a')),
\qquad a,a'\in\Z^{2p},
\end{equation}
since $W_1^t\te_{11}W_1=\te_{11}$.

To verify that $Q_F$ is well defined, use
$ZW_1=(W_1^t)^{-1}Z$. For $a\in\Ld_Z$, both
$Za$ and $(W_1^t)^{-1}Za$ are integral, so
$W_1a\in\Ld_Z$. Applying the same argument to $W_1^{-1}$
gives
\begin{equation}\label{eq:A-LambdaZ}
W_1\Ld_Z=\Ld_Z.
\end{equation}
Thus $W_1$ induces an automorphism of $\Z^{2p}/\Ld_Z$,
which is exactly $Q_F$ under Lemma~\ref{lem:finite-quotient}.
It preserves the bicharacter because
\[
\begin{aligned}
\rh_F(Q_F\phi_F(a),Q_F\phi_F(a'))
&=\exp(2\pi i a^tW_1^tZW_1a')\\
&=\exp(2\pi i a^tZa')
=\rh_F(\phi_F(a),\phi_F(a')),
\end{aligned}
\]
and $\phi_F$ is surjective.

On the real factor, $Q_{\R}=T_{11}W_1T_{11}^{-1}$ satisfies
\[
\begin{aligned}
Q_{\R}^tJ_0Q_{\R}
&=T_{11}^{-t}W_1^t(T_{11}^tJ_0T_{11})W_1T_{11}^{-1}\\
&=T_{11}^{-t}(\te_{11}-Z)T_{11}^{-1}=J_0.
\end{aligned}
\]
Therefore $Q_0=Q_{\R}\oplus Q_F$ is a continuous
symplectic automorphism of $G_0$. Its definition gives
\[
Q_0\ta(a)
=\bigl(T_{11}W_1a,\phi_F(W_1a)\bigr)
=\ta(W_1a),
\]
as asserted in~\eqref{eq:Q0-tau}.

Finally, the character in~\eqref{eq:ell-def} determines a
homomorphism $\ell:G_0\to\T^q$. Its coordinates are
continuous, since they are obtained by evaluating that character
on the standard basis of $\Z^q$. For $a\in\Z^{2p}$ and
$z\in\Z^q$,
\[
\begin{aligned}
\exp(2\pi i z^t\ell(\ta(a)))
&=\rh_0(\ta(W_1a),\ta(W_2z))\\
&=\exp(2\pi i a^tW_1^t\te_{11}W_2z)\\
&=\exp(-2\pi i z^tW_2^t\te_{11}W_1a).
\end{aligned}
\]
Characters of $\Z^q$ are determined by their values, so
\begin{equation}\label{eq:ell-lattice}
\ell(\ta(a))\equiv-W_2^t\te_{11}W_1a\pmod{\Z^q}.
\end{equation}

\subsection{}\label{app:symplecticity}

\begin{proof}[Proof of Proposition~\ref{prp:Q symplectic}]
Each component of~\eqref{eq:Q-def} is a continuous
homomorphism. To find the inverse, write
$(u',b',t')=Q(u,b,t)$ and solve successively:
\[
b=W_4^{-1}b',\qquad
u=Q_0^{-1}(u'-c(b)),\qquad
t=W_4^tt'-\ell(u)-\Xi b.
\]
These formulas are well defined and continuous because $W_4$
is unimodular and $Q_0$ is a topological automorphism.

For $x=(u,b,t)$ and $y=(v,d,s)$, write the bicharacter as
\[
\rh_G(x,y)=\rh_0(u,v)
\exp\!\bigl(2\pi i(b^ts-d^tt)\bigr).
\]
Substitution of~\eqref{eq:Q-def} gives
\[
\begin{aligned}
\rh_G(Qx,Qy)
={}&\rh_0(Q_0u+c(b),Q_0v+c(d))\\
&\quad\times\exp\!\left(2\pi i\left[
b^t(s+\ell(v)+\Xi d)-d^t(t+\ell(u)+\Xi b)
\right]\right).
\end{aligned}
\]
In the first factor, symplecticity of $Q_0$ gives
$\rh_0(Q_0u,Q_0v)=\rh_0(u,v)$, while~\eqref{eq:ell-def}
gives the mixed factors
\[
\rh_0(Q_0u,c(d))=\exp(2\pi i d^t\ell(u)),\qquad
\rh_0(c(b),Q_0v)=\exp(-2\pi i b^t\ell(v)).
\]
They cancel the corresponding factors from the torus coordinate.
Hence
\begin{equation}\label{eq:Q-rho-reduction}
\frac{\rh_G(Qx,Qy)}{\rh_G(x,y)}
=\rh_0(c(b),c(d))
\exp\!\bigl(2\pi i b^t(\Xi-\Xi^t)d\bigr).
\end{equation}
By~\eqref{eq:rho0-tau},
\[
\rh_0(c(b),c(d))
=\exp(2\pi i b^tW_2^t\te_{11}W_2d).
\]
Using $T_{32}-T_{32}^t=\te_{22}$ and
$\te_{21}^t=-\te_{12}$ in the definition of $\Xi$ yields
\[
\begin{aligned}
\Xi-\Xi^t
&=W_4^t\te_{21}W_2+W_2^t\te_{12}W_4
+W_4^t\te_{22}W_4-\te_{22}=-W_2^t\te_{11}W_2,
\end{aligned}
\]
where the last equality is~\eqref{eq:theta-block-identities}.
The two factors in~\eqref{eq:Q-rho-reduction} therefore
cancel, proving the claim.
\end{proof}

\subsection{}\label{app:embedding}

\begin{proof}[Proof of Proposition~\ref{prop:QT}]
For $(a,b)\in\Z^{2p}\oplus\Z^q$, equations
\eqref{eq:T-tau} and~\eqref{eq:Q-def} give
\[
\begin{aligned}
QT(a,b)=\bigl(&Q_0\ta(a)+c(b),\,W_4b, (W_4^t)^{-1}[\te_{21}a+T_{32}b+\ell(\ta(a))+\Xi b]\bigr).
\end{aligned}
\]
The first coordinate equals $\ta(W_1a+W_2b)$ by
\eqref{eq:Q0-tau} and the definition of $c$.
For the last coordinate, use~\eqref{eq:ell-lattice} and
\eqref{eq:theta-block-identities} to obtain
\[
\begin{aligned}
\te_{21}a+T_{32}b-W_2^t\te_{11}W_1a+\Xi b
&=(\te_{21}-W_2^t\te_{11}W_1)a+(T_{32}+\Xi)b\\
&=W_4^t\te_{21}(W_1a+W_2b)+W_4^tT_{32}W_4b.
\end{aligned}
\]
Since $(W_4^t)^{-1}$ is integral, this equality descends to
$\T^q$ and gives
\[
QT(a,b)=\bigl(\ta(W_1a+W_2b),\,W_4b,
[\te_{21}(W_1a+W_2b)+T_{32}W_4b]\bigr).
\]
The right-hand side is $T(W(a,b))$.
\end{proof}

\subsection{}\label{app:composition}

For $V=\begin{pmatrix}A_V&B_V\\0&H_V\end{pmatrix}
\in\mathcal G_{\te,Z}$, denote the automorphism on $G_0$
by $Q_{0,A_V}$, the character homomorphism by $\ell_V$, and
the matrix $\Xi$ obtained from $V$ by $\Xi_V$. These are obtained from
the construction in Section~3 by replacing
$(W_1,W_2,W_4)$ with $(A_V,B_V,H_V)$.

\begin{proof}[Proof of Lemma~\ref{lem:Q-functorial}]
Write $V_i=\begin{pmatrix}A_i&B_i\\0&H_i\end{pmatrix}$
for $i=1,2$, and put $V_{12}=V_1V_2$. Its blocks are
\[
A_{12}=A_1A_2,\qquad
B_{12}=A_1B_2+B_1H_2,\qquad
H_{12}=H_1H_2.
\]
For $i=1,2,12$, abbreviate
$Q_{0,i}=Q_{0,A_i}$, $c_i(b)=\ta(B_ib)$,
$\ell_i=\ell_{V_i}$ and $\Xi_i=\Xi_{V_i}$.
The construction gives
\begin{equation}\label{eq:QV12}
\begin{aligned}
Q_{V_{12}}(u,b,t)=\bigl(&Q_{0,12}u+c_{12}(b),\,H_{12}b, (H_{12}^t)^{-1}(t+\ell_{12}(u)+\Xi_{12}b)\bigr).
\end{aligned}
\end{equation}
On the real factor,
\[
(T_{11}A_1T_{11}^{-1})(T_{11}A_2T_{11}^{-1})
=T_{11}A_1A_2T_{11}^{-1},
\]
and on the finite factor,
$Q_{F,A_1}Q_{F,A_2}\phi_F(a)=\phi_F(A_1A_2a)$.
Surjectivity of $\phi_F$ thus proves
\begin{equation}\label{eq:Q0-composition}
Q_{0,1}Q_{0,2}=Q_{0,12}.
\end{equation}
Also,
\begin{equation}\label{eq:c-composition}
\begin{aligned}
Q_{0,1}c_2(b)+c_1(H_2b)
&=\ta(A_1B_2b)+\ta(B_1H_2b)=\ta(B_{12}b)=c_{12}(b).
\end{aligned}
\end{equation}

For $z\in\Z^q$, symplecticity of $Q_{0,1}$ and
\eqref{eq:ell-def} give
\[
\begin{aligned}
\rh_0(Q_{0,12}u,c_{12}(z))
&=\rh_0(Q_{0,1}Q_{0,2}u,Q_{0,1}c_2(z)+c_1(H_2z))\\
&=\rh_0(Q_{0,2}u,c_2(z))
\rh_0(Q_{0,1}Q_{0,2}u,c_1(H_2z))\\
&=\exp\!\left(2\pi i z^t
[\ell_2(u)+H_2^t\ell_1(Q_{0,2}u)]\right).
\end{aligned}
\]
Uniqueness of the character representative in $\T^q$ yields
\begin{equation}\label{eq:ell-composition}
\ell_{12}(u)=\ell_2(u)+H_2^t\ell_1(Q_{0,2}u)
\qquad\text{in }\T^q.
\end{equation}

It remains to account for the term $\ell_1(c_2(b))$ in
the composite. From
$\Xi_i=H_i^t\te_{21}B_i+H_i^tT_{32}H_i-T_{32}$,
substitution of the blocks of $V_{12}$ gives
\[
\Xi_{12}-\Xi_2-H_2^t\Xi_1H_2
=H_2^t(H_1^t\te_{21}A_1-\te_{21})B_2.
\]
The $(2,1)$ block of $V_1^t\te V_1=\te$ is
$B_1^t\te_{11}A_1+H_1^t\te_{21}A_1=\te_{21}$, so
\begin{equation}\label{eq:Xi-composition}
\Xi_{12}=\Xi_2+H_2^t\Xi_1H_2
-H_2^tB_1^t\te_{11}A_1B_2.
\end{equation}
The lattice identity~\eqref{eq:ell-lattice}, applied to
$V_1$ and $c_2(b)=\ta(B_2b)$, reads
\[
\ell_1(c_2(b))\equiv
-B_1^t\te_{11}A_1B_2b\pmod{\Z^q}.
\]
Consequently,
\begin{equation}\label{eq:Xi-ell-composition}
\Xi_2b+H_2^t\ell_1(c_2(b))+H_2^t\Xi_1H_2b
\equiv\Xi_{12}b\pmod{\Z^q}.
\end{equation}

By~\eqref{eq:Q0-composition} and~\eqref{eq:c-composition},
the first two coordinates of $Q_{V_1}Q_{V_2}(u,b,t)$ are
$Q_{0,12}u+c_{12}(b)$ and $H_{12}b$. Its third coordinate is
\[
\begin{aligned}
&(H_1^t)^{-1}\bigl[(H_2^t)^{-1}(t+\ell_2(u)+\Xi_2b)
+\ell_1(Q_{0,2}u+c_2(b))+\Xi_1H_2b\bigr]\\
&\quad=(H_{12}^t)^{-1}\bigl[t+\ell_2(u)+\Xi_2b
+H_2^t\ell_1(Q_{0,2}u)+H_2^t\ell_1(c_2(b))
+H_2^t\Xi_1H_2b\bigr]\\
&\quad=(H_{12}^t)^{-1}(t+\ell_{12}(u)+\Xi_{12}b),
\end{aligned}
\]
where the last equality holds in $\T^q$ by
\eqref{eq:ell-composition} and~\eqref{eq:Xi-ell-composition}.
Comparison with~\eqref{eq:QV12} proves
$Q_{V_1}Q_{V_2}=Q_{V_1V_2}$.
\end{proof}

\subsection{}\label{app:right-exact}

\begin{proof}[Proof of Proposition~\ref{prop:right-exact}]
The commutation relation for Heisenberg--Weyl operators gives
\[
 \pi(z_0)\pi(z)\pi(z_0)^*
 =\overline{\rh_G(z_0,z)}\pi(z).
\]
Since $R_{Wx}$ is a scalar multiple of $\pi(TWx)$, it follows that
\[
 \begin{aligned}
 \mathcal U R_x\mathcal U^*
 &=\ch_R(x)\pi(z_0)R_{Wx}\pi(z_0)^*\\
 &=\ch_R(x)\overline{\rh_G(z_0,TWx)}R_{Wx}
 =R_{Wx}.
 \end{aligned}
\]
Since $R_x$ is right multiplication by $U_x^\te$, this gives
\[
 \mathcal U(\xi\cdot U_x^\te)
 =(\mathcal U\xi)\cdot U_{Wx}^\te
 =(\mathcal U\xi)\cdot\af_W(U_x^\te).
\]
The identity holds for finite Fourier sums by linearity. For
$a\in A_\te^\infty$, its rapidly decreasing coefficients give an
absolutely convergent operator series on $L^2(M)$, so the same
identity holds for $a$.
\end{proof}

\subsection{}\label{app:Uk-obstruction}

\begin{proof}[Proof of Lemma~\ref{lem:Uk-obstruction}]
The function $\ch_k$ is continuous by its definition. We claim
that it is a character. Conjugate the relation
$\pi(z)\pi(w)=\overline{\bt(z,w)}\pi(z+w)$ by $\mathcal U^k$.
Using~\eqref{eq:Uk-Weyl-covariance} on both sides gives
\[
 \ch_k(z)\ch_k(w)\overline{\bt(z,w)}\pi(z+w)
 =\overline{\bt(z,w)}\ch_k(z+w)\pi(z+w).
\]
Cancelling gives $\ch_k(z+w)=\ch_k(z)\ch_k(w)$. By
Lemma~\ref{lem:rhoG-self-duality}, there is a unique $v\in G$
such that
\begin{equation}\label{eq:chi-k-v}
 \ch_k(z)=\overline{\rh_G(v,z)},\qquad z\in G.
\end{equation}
Equations~\eqref{eq:Uk-Weyl-covariance}
and~\eqref{eq:chi-k-v} show that $\mathcal U^k$ and $\pi(v)$
induce the same conjugation on every $\pi(z)$. Hence
$\pi(v)^*\mathcal U^k$ commutes with the Heisenberg--Weyl
representation. By irreducibility and Schur's lemma, there is
$\ld\in\T$ such that
\[
 \mathcal U^k=\ld\pi(v).
\]

Since $\mathcal U^k$ commutes with $R_x$, and $R_x$ is a nonzero
scalar multiple of $\pi(Tx)$, we have
\[
 \pi(v)\pi(Tx)\pi(v)^*=\pi(Tx),\qquad x\in\Z^n.
\]
The commutation relation therefore gives $\rho_G(v,Tx)=1$ for
every $x$, so $v\in D^\perp$.

Finally, $\mathcal U$ commutes with $\mathcal U^k=\ld\pi(v)$.
Its covariance relation then gives
\[
 \pi(v)=\mathcal U\pi(v)\mathcal U^*
 =\eta_{\mathcal U}(v)\pi(Qv).
\]
Thus $\pi(Qv)\pi(v)^*$ is scalar. It is a scalar multiple of
$\pi(Qv-v)$, so Lemma~\ref{lem:scalar-operator} implies $Qv=v$.
\end{proof}

\subsection{}\label{app:modified-U}

\begin{proof}[Proof of Proposition~\ref{prp:modified U}]
By Lemma~\ref{lem:Uk-obstruction}, write
$\mathcal U^k=\ld\pi(v)$ with $v\in D^\perp$ and $Qv=v$.
There is a unique $m\in\Z^n$ with $v=S(m)$. Equation~\eqref{eq: QS}
gives
\[
 S(P_1^tm)=QS(m)=Qv=v=S(m).
\]
Since $S$ is injective, $P_1^tm=m$. Now
Lemma~\ref{W and A compatibility} gives
\[
 P_1^t=\A^{-1}W\A.
\]
Thus $P_1^t$ has no eigenvalue $1$, and $m=0$. It follows that
$v=0$ and $\mathcal U^k=\ld I$.

Choose $\tilde\ld\in\T$ with $\tilde\ld^k=\ld^{-1}$ and put
$\mathcal U_0=\tilde\ld\mathcal U$. Then $\mathcal U_0^k=I$.
The right covariance relation is unchanged by this scalar
multiple, and $\mathcal U_0$ still preserves $\mathcal S(M)$.
\end{proof}

\subsection{}\label{app:delta-character}

\begin{proof}[Proof of Lemma~\ref{lem:delta-character}]
Conjugating
$\widetilde L_x\widetilde L_y=\sm_\Th(x,y)\widetilde L_{x+y}$
by $\mathcal U_0$, we obtain
\[
 \dt(x)\dt(y)\sm_\Th(P_1^tx,P_1^ty)
 \widetilde L_{P_1^t(x+y)}
 =\sm_\Th(x,y)\dt(x+y)\widetilde L_{P_1^t(x+y)}.
\]
Since $P_1\Th P_1^t=\Th$, the two multipliers agree. Cancelling
them gives $\dt(x+y)=\dt(x)\dt(y)$.
\end{proof}

\subsection{}\label{app:character-solution}

\begin{proof}[Proof of Lemma~\ref{lem:gauge-solution}]
Write $\dt(x)=\exp(2\pi i r^tx)$, with
$r\in\R^n/\Z^n$, and seek $\om$ in the form
$\om(x)=\exp(2\pi i s^tx)$. Then
\[
 \frac{\om(P_1^tx)}{\om(x)}
 =\exp\!\bigl(2\pi i((P_1-I_n)s)^tx\bigr).
\]
Hence~\eqref{eq:gauge-character-equation} is equivalent to
\begin{equation}\label{eq:gauge-linear}
 (P_1-I_n)s\equiv-r\pmod{\Z^n}.
\end{equation}
The matrix $P_1$ has no eigenvalue $1$, since $P_1^t$ is similar
to $W$. Choose a representative $r\in\R^n$ and set
\[
 s=-(P_1-I_n)^{-1}r.
\]
Then $(P_1-I_n)s=-r$ in $\R^n$, so the character
$\om(x)=\exp(2\pi i s^tx)$ has the required property.
\end{proof}

\subsection{}\label{app:left-exact}

\begin{proof}[Proof of Proposition~\ref{prop:left-exact}]
The choice of $\om$ gives
\[
 \mathcal U_0L_x^{(\om)}\mathcal U_0^*
 =L_{P_1^tx}^{(\om)},\qquad x\in\Z^n.
\]
This is~\eqref{eq:left-exact} for $b=U_x^\Th$. Linearity and
convergence of the Fourier series extend the identity to
$A_\Th^\infty$.
\end{proof}

\subsection{}\label{app:inner-product-covariance}

\begin{proof}[Proof of Proposition~\ref{prop:inner-product-covariance}]
For the right inner product, take $y\in\Z^n$. Unitarity moves
$\mathcal U_0$ from one entry of the $L^2$ pairing to the other. The right
covariance relation then changes the lattice index from $y$ to
$W^{-1}y$. More explicitly,
\[
 \begin{aligned}
 \langle\mathcal U_0f,\mathcal U_0h\rangle_{A_\te}(y)
 &=(\mathcal U_0f,R_{-y}\mathcal U_0h)_2\\
 &=(f,\mathcal U_0^*R_{-y}\mathcal U_0h)_2\\
 &=(f,R_{-W^{-1}y}h)_2
 =\langle f,h\rangle_{A_\te}(W^{-1}y).
 \end{aligned}
\]
Reindexing the Fourier series by $y=Wx$ gives
\[
 \begin{aligned}
 \langle\mathcal U_0f,\mathcal U_0h\rangle_{A_\te}
 &=\sum_{x\in\Z^n}\langle f,h\rangle_{A_\te}(x)U_{Wx}^\te\\
 &=\af_W\bigl(\langle f,h\rangle_{A_\te}\bigr).
 \end{aligned}
\]
This proves~\eqref{eq:right-ip-covariance}.

The same argument applies on the left, using
\eqref{eq:left-ip-coefficient} and the covariance of
$L_y^{(\om)}$:
\[
 \begin{aligned}
 {}_{A_\Th}\langle\mathcal U_0f,\mathcal U_0h\rangle(y)
 &=\kappa(L_y^{(\om)}\mathcal U_0h,\mathcal U_0f)_2\\
 &=\kappa(\mathcal U_0^*L_y^{(\om)}\mathcal U_0h,f)_2\\
 &=\kappa(L_{(P_1^t)^{-1}y}^{(\om)}h,f)_2\\
 &={}_{A_\Th}\langle f,h\rangle((P_1^t)^{-1}y).
 \end{aligned}
\]
Reindexing by $y=P_1^tx$ now gives
\[
 \begin{aligned}
 {}_{A_\Th}\langle\mathcal U_0f,\mathcal U_0h\rangle
 &=\sum_{x\in\Z^n}{}_{A_\Th}\langle f,h\rangle(x)U_{P_1^tx}^\Th=\af_{P_1^t}\bigl({}_{A_\Th}\langle f,h\rangle\bigr),
 \end{aligned}
\]
proving~\eqref{eq:left-ip-covariance}.

Taking $h=f$ in the right covariance identity, we get
\[
 \|\mathcal U_0f\|_E^2
 =\bigl\|\af_W(\langle f,f\rangle_{A_\te})\bigr\|
 =\|\langle f,f\rangle_{A_\te}\|
 =\|f\|_E^2.
\]
Thus $\mathcal U_0$ extends to an isometry of $E$. The identity
$\mathcal U_0^k=I$ also extends from the dense subspace
$\mathcal S(M)$, so this isometry is surjective, with inverse
$\mathcal U_0^{k-1}$. Continuity of the module actions and inner
products then extends all the covariance identities to $E$.
\end{proof}

\end{document}